\documentclass[11pt]{article}
\usepackage{xcolor}
\usepackage{amscd}
\usepackage{graphicx,epsfig}
\usepackage{xcolor}
\usepackage{comment}
\usepackage{tikz}
\usepackage{amsmath,amscd,amsthm,amsfonts,amssymb}
\usepackage{xcolor}

	\newtheorem{thm}{Theorem}
	\newtheorem*{thm*}{Theorem}
	\newtheorem{cor}{Corollary}
	\newtheorem{lem}{Lemma}
	\newtheorem{rem}{Remark}

	\newtheorem{prop}{Proposition}
	\newtheorem{defn}{Definition}
	\newtheorem{notation}{Terminology}

	\newtheorem{ex}{Example}
	
	\newcounter{constant}

	 \newcommand{\rn}{\mathbb{R}^n}
	  
		   	  \newcommand{\squatre}{\mathbb{S}^{4}}

	  \newcommand{\et}{\mathbb{E}^3}
	  
	    \newcommand{\rpdeux}{\mathbb{RP}^2}

	  \newcommand{\equat}{\mathbb{E}^4}

\def\XXint#1#2#3{{\setbox0=\hbox{$#1{#2#3}{\int}$ }
\vcenter{\hbox{$#2#3$ }}\kern-.6\wd0}}

    \usepackage[breakable]{tcolorbox}
    \usepackage{parskip} 
    
    \usepackage{iftex}
    \ifPDFTeX
    	\usepackage[T1]{fontenc}
    	\usepackage{mathpazo}
    \else
    	\usepackage{fontspec}
    \fi

    \usepackage{graphicx}
    
    \usepackage{caption}
    \DeclareCaptionFormat{nocaption}{}
    \usepackage{float}
    \usepackage{xcolor} 
    \usepackage{enumerate} 
    \usepackage{geometry} 
    \usepackage{amsmath} 
    \usepackage{amssymb} 
    \usepackage{textcomp} 
    \AtBeginDocument{%
    }
    \usepackage{upquote} 
    \usepackage{eurosym} 
    \usepackage[mathletters]{ucs} 
    \usepackage{fancyvrb} 
    \usepackage{grffile} 
    \makeatletter 
    \@ifpackagelater{grffile}{2019/11/01}
    {
    }
    {
      \def\Gread@@xetex#1{%
        \IfFileExists{"\Gin@base".bb}%
        {\Gread@eps{\Gin@base.bb}}%
        {\Gread@@xetex@aux#1}%
      }
    }
    \makeatother
    \usepackage[Export]{adjustbox} 
    \adjustboxset{max size={0.9\linewidth}{0.9\paperheight}}

    \usepackage{hyperref}
    \usepackage{titling}
    \usepackage{longtable} 
    \usepackage{booktabs}  
    \usepackage[inline]{enumitem} 
    \usepackage[normalem]{ulem} 
    \usepackage{mathrsfs}

    \definecolor{urlcolor}{rgb}{0,.145,.698}
    \definecolor{linkcolor}{rgb}{.71,0.21,0.01}
    \definecolor{citecolor}{rgb}{.12,.54,.11}

    \definecolor{ansi-black}{HTML}{3E424D}
    \definecolor{ansi-black-intense}{HTML}{282C36}
    \definecolor{ansi-red}{HTML}{E75C58}
    \definecolor{ansi-red-intense}{HTML}{B22B31}
    \definecolor{ansi-green}{HTML}{00A250}
    \definecolor{ansi-green-intense}{HTML}{007427}
    \definecolor{ansi-yellow}{HTML}{DDB62B}
    \definecolor{ansi-yellow-intense}{HTML}{B27D12}
    \definecolor{ansi-blue}{HTML}{208FFB}
    \definecolor{ansi-blue-intense}{HTML}{0065CA}
    \definecolor{ansi-magenta}{HTML}{D160C4}
    \definecolor{ansi-magenta-intense}{HTML}{A03196}
    \definecolor{ansi-cyan}{HTML}{60C6C8}
    \definecolor{ansi-cyan-intense}{HTML}{258F8F}
    \definecolor{ansi-white}{HTML}{C5C1B4}
    \definecolor{ansi-white-intense}{HTML}{A1A6B2}
    \definecolor{ansi-default-inverse-fg}{HTML}{FFFFFF}
    \definecolor{ansi-default-inverse-bg}{HTML}{000000}

    \definecolor{outerrorbackground}{HTML}{FFDFDF}

    \DefineVerbatimEnvironment{Highlighting}{Verbatim}{commandchars=\\\{\}}

    \let\Oldtex\TeX
    \let\Oldlatex\LaTeX
    \renewcommand{\TeX}{\textrm{\Oldtex}}
    \renewcommand{\LaTeX}{\textrm{\Oldlatex}}
    \title{Untitled}

\makeatletter
\def\PY@reset{\let\PY@it=\relax \let\PY@bf=\relax%
    \let\PY@ul=\relax \let\PY@tc=\relax%
    \let\PY@bc=\relax \let\PY@ff=\relax}
\def\PY@tok#1{\csname PY@tok@#1\endcsname}
\def\PY@toks#1+{\ifx\relax#1\empty\else%
    \PY@tok{#1}\expandafter\PY@toks\fi}
\def\PY@do#1{\PY@bc{\PY@tc{\PY@ul{%
    \PY@it{\PY@bf{\PY@ff{#1}}}}}}}
\def\PY#1#2{\PY@reset\PY@toks#1+\relax+\PY@do{#2}}

\@namedef{PY@tok@w}{\def\PY@tc##1{\textcolor[rgb]{0.73,0.73,0.73}{##1}}}
\@namedef{PY@tok@c}{\let\PY@it=\textit\def\PY@tc##1{\textcolor[rgb]{0.25,0.50,0.50}{##1}}}
\@namedef{PY@tok@cp}{\def\PY@tc##1{\textcolor[rgb]{0.74,0.48,0.00}{##1}}}
\@namedef{PY@tok@k}{\let\PY@bf=\textbf\def\PY@tc##1{\textcolor[rgb]{0.00,0.50,0.00}{##1}}}
\@namedef{PY@tok@kp}{\def\PY@tc##1{\textcolor[rgb]{0.00,0.50,0.00}{##1}}}
\@namedef{PY@tok@kt}{\def\PY@tc##1{\textcolor[rgb]{0.69,0.00,0.25}{##1}}}
\@namedef{PY@tok@o}{\def\PY@tc##1{\textcolor[rgb]{0.40,0.40,0.40}{##1}}}
\@namedef{PY@tok@ow}{\let\PY@bf=\textbf\def\PY@tc##1{\textcolor[rgb]{0.67,0.13,1.00}{##1}}}
\@namedef{PY@tok@nb}{\def\PY@tc##1{\textcolor[rgb]{0.00,0.50,0.00}{##1}}}
\@namedef{PY@tok@nf}{\def\PY@tc##1{\textcolor[rgb]{0.00,0.00,1.00}{##1}}}
\@namedef{PY@tok@nc}{\let\PY@bf=\textbf\def\PY@tc##1{\textcolor[rgb]{0.00,0.00,1.00}{##1}}}
\@namedef{PY@tok@nn}{\let\PY@bf=\textbf\def\PY@tc##1{\textcolor[rgb]{0.00,0.00,1.00}{##1}}}
\@namedef{PY@tok@ne}{\let\PY@bf=\textbf\def\PY@tc##1{\textcolor[rgb]{0.82,0.25,0.23}{##1}}}
\@namedef{PY@tok@nv}{\def\PY@tc##1{\textcolor[rgb]{0.10,0.09,0.49}{##1}}}
\@namedef{PY@tok@no}{\def\PY@tc##1{\textcolor[rgb]{0.53,0.00,0.00}{##1}}}
\@namedef{PY@tok@nl}{\def\PY@tc##1{\textcolor[rgb]{0.63,0.63,0.00}{##1}}}
\@namedef{PY@tok@ni}{\let\PY@bf=\textbf\def\PY@tc##1{\textcolor[rgb]{0.60,0.60,0.60}{##1}}}
\@namedef{PY@tok@na}{\def\PY@tc##1{\textcolor[rgb]{0.49,0.56,0.16}{##1}}}
\@namedef{PY@tok@nt}{\let\PY@bf=\textbf\def\PY@tc##1{\textcolor[rgb]{0.00,0.50,0.00}{##1}}}
\@namedef{PY@tok@nd}{\def\PY@tc##1{\textcolor[rgb]{0.67,0.13,1.00}{##1}}}
\@namedef{PY@tok@s}{\def\PY@tc##1{\textcolor[rgb]{0.73,0.13,0.13}{##1}}}
\@namedef{PY@tok@sd}{\let\PY@it=\textit\def\PY@tc##1{\textcolor[rgb]{0.73,0.13,0.13}{##1}}}
\@namedef{PY@tok@si}{\let\PY@bf=\textbf\def\PY@tc##1{\textcolor[rgb]{0.73,0.40,0.53}{##1}}}
\@namedef{PY@tok@se}{\let\PY@bf=\textbf\def\PY@tc##1{\textcolor[rgb]{0.73,0.40,0.13}{##1}}}
\@namedef{PY@tok@sr}{\def\PY@tc##1{\textcolor[rgb]{0.73,0.40,0.53}{##1}}}
\@namedef{PY@tok@ss}{\def\PY@tc##1{\textcolor[rgb]{0.10,0.09,0.49}{##1}}}
\@namedef{PY@tok@sx}{\def\PY@tc##1{\textcolor[rgb]{0.00,0.50,0.00}{##1}}}
\@namedef{PY@tok@m}{\def\PY@tc##1{\textcolor[rgb]{0.40,0.40,0.40}{##1}}}
\@namedef{PY@tok@gh}{\let\PY@bf=\textbf\def\PY@tc##1{\textcolor[rgb]{0.00,0.00,0.50}{##1}}}
\@namedef{PY@tok@gu}{\let\PY@bf=\textbf\def\PY@tc##1{\textcolor[rgb]{0.50,0.00,0.50}{##1}}}
\@namedef{PY@tok@gd}{\def\PY@tc##1{\textcolor[rgb]{0.63,0.00,0.00}{##1}}}
\@namedef{PY@tok@gi}{\def\PY@tc##1{\textcolor[rgb]{0.00,0.63,0.00}{##1}}}
\@namedef{PY@tok@gr}{\def\PY@tc##1{\textcolor[rgb]{1.00,0.00,0.00}{##1}}}
\@namedef{PY@tok@ge}{\let\PY@it=\textit}
\@namedef{PY@tok@gs}{\let\PY@bf=\textbf}
\@namedef{PY@tok@gp}{\let\PY@bf=\textbf\def\PY@tc##1{\textcolor[rgb]{0.00,0.00,0.50}{##1}}}
\@namedef{PY@tok@go}{\def\PY@tc##1{\textcolor[rgb]{0.53,0.53,0.53}{##1}}}
\@namedef{PY@tok@gt}{\def\PY@tc##1{\textcolor[rgb]{0.00,0.27,0.87}{##1}}}
\@namedef{PY@tok@err}{\def\PY@bc##1{{\setlength{\fboxsep}{\string -\fboxrule}\fcolorbox[rgb]{1.00,0.00,0.00}{1,1,1}{\strut ##1}}}}
\@namedef{PY@tok@kc}{\let\PY@bf=\textbf\def\PY@tc##1{\textcolor[rgb]{0.00,0.50,0.00}{##1}}}
\@namedef{PY@tok@kd}{\let\PY@bf=\textbf\def\PY@tc##1{\textcolor[rgb]{0.00,0.50,0.00}{##1}}}
\@namedef{PY@tok@kn}{\let\PY@bf=\textbf\def\PY@tc##1{\textcolor[rgb]{0.00,0.50,0.00}{##1}}}
\@namedef{PY@tok@kr}{\let\PY@bf=\textbf\def\PY@tc##1{\textcolor[rgb]{0.00,0.50,0.00}{##1}}}
\@namedef{PY@tok@bp}{\def\PY@tc##1{\textcolor[rgb]{0.00,0.50,0.00}{##1}}}
\@namedef{PY@tok@fm}{\def\PY@tc##1{\textcolor[rgb]{0.00,0.00,1.00}{##1}}}
\@namedef{PY@tok@vc}{\def\PY@tc##1{\textcolor[rgb]{0.10,0.09,0.49}{##1}}}
\@namedef{PY@tok@vg}{\def\PY@tc##1{\textcolor[rgb]{0.10,0.09,0.49}{##1}}}
\@namedef{PY@tok@vi}{\def\PY@tc##1{\textcolor[rgb]{0.10,0.09,0.49}{##1}}}
\@namedef{PY@tok@vm}{\def\PY@tc##1{\textcolor[rgb]{0.10,0.09,0.49}{##1}}}
\@namedef{PY@tok@sa}{\def\PY@tc##1{\textcolor[rgb]{0.73,0.13,0.13}{##1}}}
\@namedef{PY@tok@sb}{\def\PY@tc##1{\textcolor[rgb]{0.73,0.13,0.13}{##1}}}
\@namedef{PY@tok@sc}{\def\PY@tc##1{\textcolor[rgb]{0.73,0.13,0.13}{##1}}}
\@namedef{PY@tok@dl}{\def\PY@tc##1{\textcolor[rgb]{0.73,0.13,0.13}{##1}}}
\@namedef{PY@tok@s2}{\def\PY@tc##1{\textcolor[rgb]{0.73,0.13,0.13}{##1}}}
\@namedef{PY@tok@sh}{\def\PY@tc##1{\textcolor[rgb]{0.73,0.13,0.13}{##1}}}
\@namedef{PY@tok@s1}{\def\PY@tc##1{\textcolor[rgb]{0.73,0.13,0.13}{##1}}}
\@namedef{PY@tok@mb}{\def\PY@tc##1{\textcolor[rgb]{0.40,0.40,0.40}{##1}}}
\@namedef{PY@tok@mf}{\def\PY@tc##1{\textcolor[rgb]{0.40,0.40,0.40}{##1}}}
\@namedef{PY@tok@mh}{\def\PY@tc##1{\textcolor[rgb]{0.40,0.40,0.40}{##1}}}
\@namedef{PY@tok@mi}{\def\PY@tc##1{\textcolor[rgb]{0.40,0.40,0.40}{##1}}}
\@namedef{PY@tok@il}{\def\PY@tc##1{\textcolor[rgb]{0.40,0.40,0.40}{##1}}}
\@namedef{PY@tok@mo}{\def\PY@tc##1{\textcolor[rgb]{0.40,0.40,0.40}{##1}}}
\@namedef{PY@tok@ch}{\let\PY@it=\textit\def\PY@tc##1{\textcolor[rgb]{0.25,0.50,0.50}{##1}}}
\@namedef{PY@tok@cm}{\let\PY@it=\textit\def\PY@tc##1{\textcolor[rgb]{0.25,0.50,0.50}{##1}}}
\@namedef{PY@tok@cpf}{\let\PY@it=\textit\def\PY@tc##1{\textcolor[rgb]{0.25,0.50,0.50}{##1}}}
\@namedef{PY@tok@c1}{\let\PY@it=\textit\def\PY@tc##1{\textcolor[rgb]{0.25,0.50,0.50}{##1}}}
\@namedef{PY@tok@cs}{\let\PY@it=\textit\def\PY@tc##1{\textcolor[rgb]{0.25,0.50,0.50}{##1}}}

\makeatother

    \makeatletter
        \newbox\Wrappedcontinuationbox 
        \newbox\Wrappedvisiblespacebox 
        \newcommand*\Wrappedvisiblespace {\textcolor{red}{\textvisiblespace}} 
        \newcommand*\Wrappedcontinuationsymbol {\textcolor{red}{\llap{\tiny$\m@th\hookrightarrow$}}} 
        \newcommand*\Wrappedcontinuationindent {3ex } 
        \newcommand*\Wrappedafterbreak {\kern\Wrappedcontinuationindent\copy\Wrappedcontinuationbox} 
        \newcommand*\Wrappedbreaksatspecials {%
            \def\PYGZus{\discretionary{\char`\_}{\Wrappedafterbreak}{\char`\_}}%
            \def\PYGZob{\discretionary{}{\Wrappedafterbreak\char`\{}{\char`\{}}%
            \def\PYGZcb{\discretionary{\char`\}}{\Wrappedafterbreak}{\char`\}}}%
            \def\PYGZca{\discretionary{\char`\^}{\Wrappedafterbreak}{\char`\^}}%
            \def\PYGZam{\discretionary{\char`\&}{\Wrappedafterbreak}{\char`\&}}%
            \def\PYGZlt{\discretionary{}{\Wrappedafterbreak\char`\<}{\char`\<}}%
            \def\PYGZgt{\discretionary{\char`\>}{\Wrappedafterbreak}{\char`\>}}%
            \def\PYGZsh{\discretionary{}{\Wrappedafterbreak\char`\#}{\char`\#}}%
            \def\PYGZpc{\discretionary{}{\Wrappedafterbreak\char`\%}{\char`\%}}%
            \def\PYGZdl{\discretionary{}{\Wrappedafterbreak\char`\$}{\char`\$}}%
            \def\PYGZhy{\discretionary{\char`\-}{\Wrappedafterbreak}{\char`\-}}%
            \def\PYGZsq{\discretionary{}{\Wrappedafterbreak\textquotesingle}{\textquotesingle}}%
            \def\PYGZdq{\discretionary{}{\Wrappedafterbreak\char`\"}{\char`\"}}%
            \def\PYGZti{\discretionary{\char`\~}{\Wrappedafterbreak}{\char`\~}}%
        } 
        \newcommand*\Wrappedbreaksatpunct {%
            \lccode`\~`\.\lowercase{\def~}{\discretionary{\hbox{\char`\.}}{\Wrappedafterbreak}{\hbox{\char`\.}}}%
            \lccode`\~`\,\lowercase{\def~}{\discretionary{\hbox{\char`\,}}{\Wrappedafterbreak}{\hbox{\char`\,}}}%
            \lccode`\~`\;\lowercase{\def~}{\discretionary{\hbox{\char`\;}}{\Wrappedafterbreak}{\hbox{\char`\;}}}%
            \lccode`\~`\:\lowercase{\def~}{\discretionary{\hbox{\char`\:}}{\Wrappedafterbreak}{\hbox{\char`\:}}}%
            \lccode`\~`\?\lowercase{\def~}{\discretionary{\hbox{\char`\?}}{\Wrappedafterbreak}{\hbox{\char`\?}}}%
            \lccode`\~`\!\lowercase{\def~}{\discretionary{\hbox{\char`\!}}{\Wrappedafterbreak}{\hbox{\char`\!}}}%
            \lccode`\~`\/\lowercase{\def~}{\discretionary{\hbox{\char`\/}}{\Wrappedafterbreak}{\hbox{\char`\/}}}%
            \catcode`\.\active
            \catcode`\,\active 
            \catcode`\;\active
            \catcode`\:\active
            \catcode`\?\active
            \catcode`\!\active
            \catcode`\/\active 
            \lccode`\~`\~ 	
        }
    \makeatother

    \let\OriginalVerbatim=\Verbatim
    \makeatletter
    \renewcommand{\Verbatim}[1][1]{%
        \sbox\Wrappedcontinuationbox {\Wrappedcontinuationsymbol}%
        \sbox\Wrappedvisiblespacebox {\FV@SetupFont\Wrappedvisiblespace}%
        \def\FancyVerbFormatLine ##1{\hsize\linewidth
            \vtop{\raggedright\hyphenpenalty\z@\exhyphenpenalty\z@
                \doublehyphendemerits\z@\finalhyphendemerits\z@
                \strut ##1\strut}%
        }%
        \def\FV@Space {%
            \nobreak\hskip\z@ plus\fontdimen3\font minus\fontdimen4\font
            \discretionary{\copy\Wrappedvisiblespacebox}{\Wrappedafterbreak}
            {\kern\fontdimen2\font}%
        }%
        
        \Wrappedbreaksatspecials
        \OriginalVerbatim[#1,codes*=\Wrappedbreaksatpunct]%
    }
    \makeatother

    \definecolor{incolor}{HTML}{303F9F}
    \definecolor{outcolor}{HTML}{D84315}
    \definecolor{cellborder}{HTML}{CFCFCF}
    \definecolor{cellbackground}{HTML}{F7F7F7}
    
    \makeatletter
    \newcommand{\boxspacing}{\kern\kvtcb@left@rule\kern\kvtcb@boxsep}
    \makeatother

    \hypersetup{
      breaklinks=true,  
      colorlinks=true,
      urlcolor=urlcolor,
      linkcolor=linkcolor,
      citecolor=citecolor,
      }
    
\begin{document}

 \title{\bf \Large    Biharmonic     Hypersurfaces in Euclidean Space  }

	\title{\bf \fontsize{12}{12} \selectfont     BIHARMONIC HYPERSURFACES IN EUCLIDEAN SPACES }
	\author{ \fontsize{10}{10} \selectfont    HIBA BIBI,  MARC SORET AND MARINA VILLE}
	\date{}

\maketitle

		\begin{abstract} 
An isometric immersion  $X: \Sigma^n  \longrightarrow  \mathbb{E}^{n+1}$ is  biharmonic if
   $\Delta^2 X = 0$, i.e.  if  $\Delta H =0$, where  $\Delta$ and  $H$ are
  the metric Laplacian   and the mean curvature  vector field of  $\Sigma^n$  respectively.  	
  More generally, biconservative hypersurfaces (BCH) are isometric immersions for which only the tangential part of the biharmonic equation vanishes.
  We study and construct BCH that are holonomic, i.e. for which the principal curvature directions define an integrable net, and we deduce that 
  $\Sigma^n$ is  a holonomic biharmonic hypersurface if and only if it is minimal.
 %
		
 	\end{abstract}
\let\thefootnote\relax\footnotetext{ \hskip -.15 in  2024 { \it Mathematics Subject Classification}. Primary 53B25 53C42  58E20\\
keywords: Hypersurfaces, Biharmonic, Biconservative,
Holonomic,
Constant Mean Curvature, Isoparametric Submanifolds.}


\section{Introduction}

An isometric immersion in a Euclidean space, $X:\Sigma^n \longrightarrow \mathbb{E}^{N}$, is  biharmonic if
   $\Delta^2 X = 0$, i.e. if  $\Delta H =0$, where  $\Delta$ and  $H$ are   
  the metric Laplacian   and the mean curvature vector field   of  $\Sigma^n$ respectively.\\
The immersion $X$  is a  critical point of a "bienergy" functional (see \cite{EL}).
  The image of such an immersion is also  called a biharmonic submanifold.
 B.Y. Chen conjectured in 
\cite{C0} that  a biharmonic submanifold  of  the Euclidean space must be minimal. 
In  this direction, a long range argument   in  \cite{AM}   proves B.Y. Chen's conjecture 
  for complete, topologically proper, biharmonic submanifolds. The proof  
uses  the maximum principle  at infinity on a complete submanifold without singularities   applied to the square of the mean curvature   function $|H|^2$.
We will be more interested in  the case of hypersurfaces and we will study small range properties of biharmonic hypersurfaces (BHH), as well as properties of biconservative hypersurfaces (BCH). The BCH hypersurfaces are less restrictive and are defined as critical points of the bienergy functional with respect to tangential deformation, i.e. for which $(\Delta H)^\top=0$.
\\
 Minimal hypersurfaces  are trivially BHH, thus we are interested in studying BHH which are not minimal, called proper BHH.
 Many works were carried out proving the non-existence of proper BHH in Euclidean spaces, hence confirming the B.Y. Chen's conjecture. In low dimensions, the B.Y. Chen's conjecture was proved in $\et$ (\cite{Chen-Ishikawa} and \cite{Jiang}), in 
 $\equat$  \cite{HV},  in $\mathbb{E}^5$ \cite{F},  and recently in $\mathbb{E}^6$  \cite{Fu}.
 \\
Constant mean curvature (CMC) hypersurfaces in a Euclidean space are trivially biconservative. Thus, we look for BCH which are not CMC, called proper BCH. An exhaustive description   of all proper BCH  in $\et$ was given  in \cite{CMOP}, and then generalized  to the existence of 
a family of $SO(p)$-symmetric BCH of $\mathbb{E}^p$   \cite{MOR}.  Also in  \cite{MOR} the authors  constructed  proper BCH   
which are $SO(p)\times SO(q)$-symmetric    in $\mathbb{E}^{p+q}$,  and they made a global study on these hypersurfaces.\\
This paper studies   the short range  behavior of  BCH and BHH   that are holonomic, i.e.  if the principal curvature directions define an integrable net, i.e.
in other words  if the principal curvature lines   define on an open subdomain  a coordinate system (cf. for example \cite{DT} and \cite{Re}).  \\
 More precisely,  we   focus on a new construction of local   proper BCH. All the known  examples   so far are locally fibered by spheres or products of spheres, which are isoparametric submanifolds. We also construct  proper BCH which    are fibered over any given  isoparametric submanifold of
 codimension $2$ in $\mathbb{E}^{n+1}$. We recall that, according to the definition  in \cite{T},
 a  submanifold is isoparametric if its normal bundle is flat, and if the principal curvatures in the direction of any  parallel normal vector field is constant (see for  example \cite{T}).\\
 \\
 In this article we prove the following main results: 
\begin{thm}\label{letheoreme1} A   holonomic proper BCH  $\Sigma^n $ in $ \mathbb{E}^{n+1}$ is  foliated by 
the level sets of the mean curvature function  $h$ of $\Sigma^n$ which  are isoparametric codimension $2$ submanifolds  of $\mathbb{E}^{n+1}$ in 
a neighbourhood  of a regular point of $h$. The symmetries of the leaves extend to symmetries of $\Sigma^n$. The integral curves of the gradient field $\nabla h$ on $\Sigma^n$ are congruent planar curves 
governed by an ODE of order two.
Conversely,  for any codimension $2$ isoparametric submanifold $ U_0\subset \mathbb{E}^{n+1}$ which is holonomic, there locally  
exists a  
proper BCH  obtained by a local normal evolution of $U_0$ in the normal bundle of $U_0$ that preserves the fibers of the normal bundle of 
$U_0$ and the symmetry group of $U_0$. 
\end{thm}

We  deduce   from  Theorem 1 
\begin{cor}
Holonomic proper BCH of $\mathbb{E}^{n+1} $    are  either foliated by round spheres $\mathbb{S}^{n-1}$,  by  products of round spheres
 $\mathbb{S}^{p}\times \mathbb{S}^{q}$, by  cylinders   $\mathbb{S}^{p}\times \mathbb{R}^q$, where $p+q=n-1$, or  by
  $\mathbb{S}^{p}\times \mathbb{S}^{q}\times\mathbb{R}^r$, where  $p+q+r=n-1$.
\end{cor}

\begin{thm}\label{theoreme2}
Holonomic BHH of $\mathbb{E}^{n+1}$ are minimal. \\
\end{thm}

 \section{Generalities on ${m}$-harmonic maps}
  \subsection{Notations} 

Since our results are of local nature, we can identify isometrically $\Sigma^n$ with a local open subset  $\Omega$ of $\mathbb{R}^{n}$. 
 Let  
 $X:\Omega \subset \rn \longrightarrow \Sigma^n\subset \mathbb{E}^{n+1}$ be  a smooth  immersion 
 where  the image $\Sigma^n$ is an  open   hypersurface   of $\mathbb{E}^{n+1}$ and $\Omega$ is  provided with 
 the pull-back metric.\\
Let   $\Delta_0 $ denote the flat Laplacian on $\Omega$ : $\Delta_0 = \sum_{i=1}^n \partial^2_{x_i}$. \\
 By definition a function $f$ defined on  $\Omega$ is harmonic if  
  $\Delta_0 f= 0,$ and  it is $m$-harmonic if $\Delta_0^m f=0,$ where
 $$\Delta_0^m := \overbrace{\Delta_0\circ\Delta_0\circ \cdots \circ \Delta_0}^m.$$
More generally,  a  map $X$ is  harmonic if $\Delta_0 X= 0$, i.e. $\Delta_0 X^i = 0$, for all $ i=1,\ldots,  n+1 $. Thus, $X$ is harmonic if and only if 
all the  coordinate functions 
that define  $X$ are harmonic. 
By extension, $X$  is $m$-harmonic if and only if each coordinate function of the position vector $X$ is  $m$-harmonic.

\begin{ex} When $\Sigma^n$ is a graph of a function $u$  over a hyperplane $\mathbb{R}^n$ in $\mathbb{R}^{n+1}$  given by the parametrization  
$$X: (x_1,\ldots,x_n) \mapsto (x_1,\ldots, x_n, u),
$$
 $\Delta_0 X=0$ if and only if the following $n+1$ linear PDEs are satisfied 
$$
\Delta_0 x_1=0,\ldots, \Delta_0 x_n=0 \quad \text{and} \quad \Delta_0 u =0.
$$
The first n-equations are always  satisfied hence $X$ is harmonic   if and only if  $u$ is a harmonic function.
 These remarks extend to the $m$-harmonic case: $X$ is $m$-harmonic  if and only if  $\Delta_0^m u =0.$ \end{ex}
Suppose now that  $X=(X^1,\ldots, X^{n+1})$  is an isometric immersion.
Let    $g= g_{ij}$ be the induced metric on $\Omega$ and $|g|$ its determinant. The Laplacian of $X$ for this metric, which is  denoted by $\Delta$, 
is equal to the vector-valued function
 $$(\Delta X)^i=\Delta X^i = Tr(Hess(X^i)) \quad \forall i =1,\ldots, n+1.$$
 By definition  $X$ is  then   an $m$-harmonic isometric immersion  if and only if  $\Delta^m X =0$. 
Equivalently, we say that  $\Sigma^n:=X(\Omega)$ is  an $m$-harmonic hypersurface of $\mathbb{R}^{n+1}$.
\\
Let us recall  that  the  Laplacian in terms of the metric $"g"$ of a smooth function $f:\Omega \longrightarrow \mathbb{R}$,
using   Einstein's summation,  is given by 
\begin{equation}\label{lap}
 \Delta f:=  \frac{1}{\sqrt{|g|}} \left(\sqrt{|g|}g^{ij}f_{,i}\right)_{,j}  =0, \quad i,j = 1,\ldots, n,
 \end{equation}
  where $f_{,i}$ denotes the derivative of $f$ with respect  to each coordinate $x_i$, for all $i=1,\ldots, n$.
   \subsection{ Remarks on the equation of a biharmonic graph} 
 Let us find the PDE equation of a biharmonic graph. Before that, let us recall  the $1$-harmonic graph equation.
 Consider, as in the  former example, the parametrization
$$X: (x_1,\ldots,x_n) \mapsto (x_1,\ldots,x_n,u)$$
 hence, $X_{,i}:= \frac{\partial X}{\partial  x_i}= (0,\ldots,0, 1,\ldots,0,u_{,i}).$ 
The induced metric of a graph on the domain  of $X$ is then given by a matrix whose coefficients are in terms of the  derivatives of $u$ :
$$g_{ij} = \delta_{ij} +u_{,i}u_{,j}\quad {\rm and} \quad  g^{ij} =  \delta_{ij} -\frac{u_{,i}u_{,j}}{W^2},$$
where $(g^{ij}) $ is the inverse matrix of $(g_{ij})$ and 
$W:=\sqrt{1+|\nabla u |^2} = \sqrt{|g|}$.
The second fundamental  form and the shape operator of $X$ are  respectively given by the following  matrices : 
$$II_{ij} = \frac{u_{,ij}}{W}\quad {\rm and}\quad A_{ij} = ^t g^{-1}II =\frac{u_{,ij}}{W}-\sum_{\alpha =1}^n\frac{u_{,i\alpha} u_{,\alpha} u_{,j}}{W^3},$$
where $i,j=1,\ldots, n$.
\\
For example, in dimension $2$, the metric and the second fundamental form are
\begin{equation}
g =\left( 
\begin{array}{cc}
1+u_{,x}^2 & u_{,x}u_{,y} \\
u_{,x}u_{,y} & 1+u_{,y}^2
\end{array}
\right) \quad {\rm and}\quad
II =\frac{1}{W}\left( 
\begin{array}{cc}
u_{,xx} & u_{,xy} \\
u_{,xy} & u_{,yy}
\end{array}
\right).
\end{equation}\\
The mean curvature vector field of the graph 
$$
H= \frac{\Delta X}{n}=\frac{1}{n} (\Delta x_1,\ldots , \Delta x_n, \Delta u)
$$
decomposes into a horizontal part: $nH_{hor} = (\Delta x_1,\ldots ,\Delta x_n, 0)$  and a vertical part:
$nH_{ver}=(0,\ldots , 0 , \Delta u)$.
A harmonic graph satisfies   $H_{hor}=0$  and $H_{ver}=0$. Using equation \eqref{lap},  for $H_{hor}=0$ and  $H_{ver}=0$, 
 we obtain, using Einstein's summation,    the following order two quasi-linear PDEs:
\begin{equation}\label{lap2}  
\left\{
\begin{array}{ll}
nH_{hor}^l =\frac{1}{\sqrt{|g|}} \left(\sqrt{|g|}g^{kl}\right)_{,k}  =0,\quad  l=1,\ldots, n \\
nH_{ver}^{n+1} =\Delta u = g^{kl} u_{,kl} =0.
\end{array}
\right.
\end{equation}
If we substitute  the value of $g^{ij}$ in terms of $u$ in the second equation we obtain the so-called minimal graph equation 
 which can also be expressed as :
$$\left(\frac{u_{,i}}{W}\right)_{,i}=0= {\rm div}\left( \frac{\nabla u}{\sqrt{1+|\nabla u|^2}}\right).$$
For example, in dimension $2$
$$ (1+u_{,y}^2)u_{,xx} -2u_{,x}u_{,y} u_{,xy} +(1+u_{,x}^2)u_{,yy}=0.$$
In other terms a minimal surface is the image of an isometric immersion which is harmonic.\\
We similarly proceed  for  biharmonic graphs. 
 First, we deduce from \eqref{lap}  that the bilaplacian  applied to a function $f$ is:
\begin{equation}\label{lap3}
\begin{array}{c}
\Delta^2 f =g^{ij}\left( g^{kl} f_{,kl} + \frac{1}{\sqrt{|g|}} \left(\sqrt{|g|}g^{kl}\right)_{,k} f_{,l} \right)_{,ij} + \\
 \frac{1}{\sqrt{|g|}} \left(\sqrt{|g|}g^{ij}\right)_{,i} \left( g^{kl} f_{,kl} + \frac{1}{\sqrt{|g|}} \left(\sqrt{|g|}g^{kl}\right)_{,k} f_{,l} \right)_{,j}.
 \end{array}
\end{equation}
The biharmonic equation $\Delta^2X=\Delta H=0$ decomposes 
into a horizontal part:  $n(\Delta H)_{hor} = (\Delta^2 x_1,\ldots ,\Delta^2 x_n, 0)$  and a vertical part:
$n(\Delta H)_{ver} =(0,\ldots , 0 , \Delta^2 u)$. A straightforward computation yields the following  fourth order PDEs:
 \begin{equation}\label{lap5}
n(\Delta H)^m_{hor}  =g^{ij}\left(    \frac{1}{\sqrt{|g|}} \left(\sqrt{|g|}g^{km}\right)_{,k}  \right)_{,ij} +
 \frac{1}{\sqrt{|g|}} \left(\sqrt{|g|}g^{ij}\right)_{,i} \left(    \frac{1}{\sqrt{|g|}} \left(\sqrt{|g|}g^{km}\right)_{,k}   \right)_{,j}=0
\end{equation}
for $m = 1,\ldots, n$.
 Using \eqref{lap5} we obtain
 \begin{equation}
\begin{array}{c}
n(\Delta H)_{ver}^{n+1} =  g^{ij}\left( (g^{kl} u_{,kl})_{,ij} + \frac{1}{\sqrt{|g|}} \left(\sqrt{|g|}g^{kl}\right)_{,k} u_{,lij} \right)+ \\
2g^{ij}\left(  \frac{1}{\sqrt{|g|}} \left(\sqrt{|g|}g^{kl}\right)_{,ik} u_{,lj} \right)+ \\
 \frac{1}{\sqrt{|g|}} \left(\sqrt{|g|}g^{ij}\right)_{,i} \left( g^{kl} u_{,klj} +\frac{1}{\sqrt{|g|}} \left(\sqrt{|g|}g^{kl}\right)_{,k} u_{,jl} \right)=0,
 \end{array}
\end{equation}
which is a semi-linear PDE  of order $4$ (its highest order term being  $g^{ij}g^{kl} u_{,klij}$). 

\section{     Biharmonic Hypersurfaces in $\mathbb{E}^{n+1}$}
Let us determine the PDEs that characterize  biharmonic hypersurfaces in $\mathbb{E}^{n+1}$
(cf. for example  \cite{C1} for more general cases).\\
 Let $X$ be any  isometric immersion :
  $$X: \Omega \subset \mathbb{R}^n \longrightarrow X(\Omega)= \Sigma^n \subset \mathbb{E}^{n+1}.$$
  We denote the mean curvature  vector field by $H$, the mean curvature function by $h$, the unit normal vector field by $N$, the shape operator by $A$, and the Laplacian operator by $\Delta$.  \\
It is easy to show that the  position vector field $X$ and the mean curvature vector  field $H$, where $H=hN$,  are related by the following Beltrami equation 
\begin{equation}\label{eq1}
\Delta X =nH.
\end{equation}
One can similarly show that 
\begin{equation}\label{eq 2}
\Delta N  =-|A|^2N - \nabla Tr A,
\end{equation}
where $Tr A= nh$. This equation decomposes the Laplacian of the unit normal vector field with respect to its  normal component ($|A|^2N$) and its tangential component ($\nabla Tr A$).
\\
  $X$ is biharmonic if and only if  
  $$
  \Delta H =\Delta	(hN) =0,
  $$ thus 
  \begin{equation}\label{Laplacian N}
    (\Delta h) N + h \Delta N + 2 \sum_{i=1}^n\nabla_{e_i} h\nabla_{e_i} N =  (\Delta h) N + h \Delta N + 2 \nabla_{\nabla h} N = 0.
  \end{equation}
 Now, from equations
  \eqref{eq1}, \eqref{eq 2} and \eqref{Laplacian N} we deduce that $\Sigma^n$ is biharmonic if and only if  it  satisfies the  following two  semi-linear PDEs of 
  orders $3$ and $4$, respectively:
  \begin{equation}\label{System3}
\left\{
\begin{array}{cc} 
 (\Delta H)^\top := 2 \nabla_{\nabla h} N -nh\nabla h   &=0\\
 (\Delta H)^\perp :=  \Delta  h -|A|^2 h  &=0.\\
\end{array}
\right.\end{equation}
 The solutions of the  first equation ($(\Delta H)^\top=0$)   correspond
to  critical points  of the bienergy functional for  compact deformations that are tangent  to $\Sigma^n$. 
 \begin{notation}
A hypersurface satisfying $(\Delta H)^\top=0$  will be denoted by BCH which stands for "Biconservative Hypersurface",  and a hypersurface satisfying  both equations in \eqref{System3} will be denoted by BHH  which stands for "Biharmonic  Hypersurface".
\end{notation}
 \begin{notation}If $\Sigma^n$ has  constant mean curvature (CMC), then $\Sigma^n$ is a BCH. Furthermore, a CMC hypersurface satisfies the  normal equation in \eqref{System3} if and only if  $h=0$. Thus, any minimal hypersurface $\Sigma^n$ is a BHH. We will later  denote by {\bf proper BCH }
 (resp. {\bf proper BHH}) a BCH which is not CMC  \cite{MOR} (resp. a BHH which is not minimal  \cite{MO}).
\end{notation}
   \subsection{Biconservative  equation in a principal curvature frame}
  At a point $p\in \Sigma^n$, we have $s$ distinct eigenvalues and $s$ corresponding eigenspaces of the shape operator $A_N$. This defines $s$ plane distributions $T_i,$ where 
   $T_i = ker(A_N(p)-\lambda_i Id)$ and  $\lambda_i $ 
   is the corresponding eigenvalue, or principal curvature, of multiplicity $rk(T_i)=m_i$
    such that $T_p\Sigma^n  = \oplus_{i=1}^s T_i$, where  $\sum_{i=1}^s m_i = n$.
  We will select  an open subdomain  where  the multiplicities $m_i,\   i=1,\ldots, s,$ of the corresponding eigenvalues  are constant. 
We will obtain  a splitting of the tangent bundle $T\Sigma^n$  into  a family of integrable distributions
  $T_i\subset T\Sigma^n $ . 
\\
  We now express the biharmonicity  in terms of a principal curvature frame.

    \begin{prop}\label{biTan}
   Let $\{e_i\}_{i=1}^n$ be a local orthonormal frame field of principal directions on $\Sigma ^n$. Then the  BCH equation 
   $$
   2 \nabla_{\nabla h} N -nh\nabla h=0
   $$
is equivalent to the following system of   of  algebraic-differential equations:
  \begin{equation}\label{leSysteme}  
\nabla_{e_i}h  = 0  \quad {\rm or}\quad n h + 2 \lambda_i = 0,  \quad i =1,\ldots,n
\end{equation}
where   $\lambda_i ,i =1,\ldots,n$ are the principal curvatures. 
    \end{prop}

\begin{proof}
Let $\{e_i\}_{i=1}^n$ be a local orthonormal frame field  of   $\Sigma^n$  where   $e_i$ is 
 the principal direction of principal curvature $\lambda_i$,  $i =1,\ldots,n$. Then
   $$2 \nabla_{(\nabla_{e_i} h) e_i} N -nh(\nabla_{e_i} h)e_i = -2\nabla_{e_i}h\left(2\lambda_i   + nh    \right) e_i =0.$$
By continuity and by reducing the domain, the system given in (\ref{leSysteme}) holds,
thus we conclude.
      \end{proof}
     \begin{defn}\label{class}
Let  $\Sigma^n$  be a BCH in $\mathbb{E}^{n+1}$. $\Sigma^n$ is  of rank  $k$  if $rk( T\Sigma^n({-\frac{nh}{2}}) )=k,$ 
where $T\Sigma^n({-\frac{nh}{2}})$ is the eigenspace of 
 the shape operator of eigenvalue ${-\frac{nh}{2}}$ in $T\Sigma^n$.
\end{defn}
 $\nabla h$ is null  when projected onto  $T\Sigma^n({-\frac{nh}{2}}) ^\perp$.
As we will see later, the rank of $\Sigma^n$ is either $0$ or $1$.

\section{Reducible biconservative and biharmonic hypersurfaces}\label{reducible}
There is a simple automatic way to construct  BCH (or BHH)  from 
BCH (BHH)  submanifolds of lower dimensions by simply constructing cylinders with  BCH (or BHH) basis.
\begin{prop}
If $\Sigma^n\subset \mathbb{E}^{n+1}$ is a BCH  (resp. a BHH) of rank $k$ 
then the Riemann product $\tilde\Sigma:= \Sigma^n\times \mathbb{E}^l \subset \mathbb{E}^{n+l+1}$
 is a BCH (resp. a BHH) of rank  $k$.
\end{prop}
\begin{proof} 
If $h=0$ then $\Sigma^n$ is minimal and so is $\Sigma^n\times \mathbb{E}^l$.
Suppose that $h\neq 0$, then the Riemann product $\tilde \Sigma$ has a  mean curvature 
$\tilde h = \frac{n}{n+l}h $ and
$rk(  T\tilde\Sigma ({-\frac{(n+l)\tilde h}{2}})) = rk( T\tilde\Sigma (-\frac{n h}{2}))=rk( T\Sigma^n (-\frac{nh}{2}))=k. $
 Furthermore, $\nabla h$ is null on $T\Sigma^n(-\frac{n h}{2})^\perp \oplus  \mathbb{E}^l$.
Hence, it is biconservative  from Proposition \ref{biTan} and of rank $k$.
If $\Sigma^n$ is a BHH, then the  second equation of  the system \label{System2} holds also for $\tilde\Sigma$ since $|A|^2 = |\tilde A|^2$. 
\end{proof}
BCH (resp. BHH) which are Riemannian products  will be called {\sl reducible}. Thus, from now on we will mainly   investigate irreducible BCH.
\section{Geometry   of proper  holonomic  BCH in $\mathbb{E}^{n+1}$}
BCH are locally solutions of 
 the system of  algebraic-differential equations which is given by Proposition \ref{biTan}. In fact,  we will show that   BCH  are either  of rank (i.e. CMC) or of rank $1$. In the latter case, a BCH is foliated by level hypersurfaces defined by  the mean curvature function $h$. We  will then study the extrinsic geometry of these leaves and that of the integral curves which are
 orthogonal to the leaves.\\
 Notice that Lemmas \ref{para1}, \ref{para3} and \ref{para4}, are similar 
 to Lemmas 2.2, 2.3 and 2.5 of \cite{HV}, but our proofs are different and based on the holonomic condition.
This condition in turn allows us  to prove  lemmas \ref{paraODE} and \ref{para5} which were not known. Also, the fact that  $rk(\lambda_n)=1$,
in Lemma  \ref{para1},
was already proven in \cite{D}.
   \subsection{  Foliation of  a BCH by the mean curvature function}

\begin{lem}\label{para1} 
A BCH  $\Sigma^n\subset \mathbb{E}^{n+1}$ is either any CMC,  i.e.  of rank $0$, or it is a proper BCH of rank  $1$. 
In the later case,  $\Sigma^n$ is 
in a neighbourhood of regular points of $h$,  foliated by  leaves, $U_t = h^{-1}(t)$,  which are the level hypersurfaces 
 of  the mean curvature function $h$  on $\Sigma^n$.
These $U_t$ are, locally, codimension 2 submanifolds of 
  $\mathbb{E}^{n+1}$  with  flat   normal bundle. The orthogonal curves to $U_t$, which are  integral curves of the vector field 
  $\nabla h$ on $\Sigma^n$, are curvature lines of principal curvature $\lambda_n = -\frac{nh}{2}$. Any other 
  principal curvature  satisfies $\lambda_i \neq -\frac{nh}{2},$ $  i =1,\ldots, n-1 $.
   \end{lem}

   \begin{proof}
When the rank of the BCH is null,  $h$ is trivially  constant.
Hence, the BCH is CMC  and from the normal equation of  \eqref{System2}
a corresponding BHH 
 automatically satisfies $h=0 $, i.e.,  $\Sigma^n$ is minimal.
 \\ 
We suppose now  that the rank  $k$  of $\Sigma^n$ is at least two.
   $\Sigma^n$ is, in a neighbourhood  of regular points of $h$,  foliated by hypersurfaces of $\Sigma^n$ which are codimension $2$ submanifolds of $\mathbb{E}^{n+1}$, and which are given by
   $$
   U_t := \{p\in \Sigma^n : h(p) = t \} .
   $$
  From Proposition \ref{biTan},
  $T\Sigma^n(-\frac{nh}{2})^\perp \subset TU_t\subset T\Sigma^n$, hence 
  $TU_t^\perp \subset T\Sigma^n(-\frac{nh}{2}) \subset T\Sigma^n$.
  Thus, the normal curves to the foliation on $\Sigma^n$ are  integral curves of 
  $$e_n:= \frac{ \nabla h}{|\nabla h|}$$
and are curvature lines of the principal curvature $\lambda_n = -\frac{nh}{2}$.\\
When the rank $k\geq 2$, the distribution defined by the $k$-planes, $T\Sigma^n(-\frac{nh}{2})$, 
  is integrable   into  a submanifold of dimension $k$ 
that has constant principal curvature $\lambda_n = -\frac{nh}{2}$ in the normal  direction  $N$ (see Appendix  \ref{spheres}). 
In particular  $\lambda_n$ is constant along the $\nabla h$ integral curves of $\nabla h$.
and so is $h$; but $h$ is also constant on the $U_t$, hence $h$ is constant everywhere which contradicts the hypothesis.
Thus, we conclude that  $rk(T\Sigma^n(-\frac{nh}{2}) )=1$.\\
 Let us show that the normal bundle $NU_t\subset \mathbb{E}^{n+1}$  is flat. Along the principal directions $\{e_i\}_{i=1}^{n-1},$ in $TU_t$, we have
$$\langle D_{e_i} N, e_n \rangle = -\langle \lambda_i e_i, e_n \rangle  =0 \quad {\rm and} \quad
\langle D_{e_i} e_n, N\rangle = -\langle  e_n, \nabla_{e_i} N \rangle  =0,$$

where $D$ is the connection  in $\mathbb{E}^{n+1}$.
Hence, the normal bundle is  flat and  its  normal curvature is null and using the Ricci equation we get
$$\langle R(X,Y)\xi , \tau \rangle = \langle [A_\xi,A_\tau ]X,Y \rangle =0,\ \forall X,Y \in TU_t, \ \forall \xi,\tau \in NU_t.$$
So the shape operators $A_N$ and $A_{e_n}$ commute. Hence $A_{e_n}$ 
is diagonalised in the same frame $\{ e_i\}_{i=1}^{n-1}$, i.e. there exist  principal curvatures  $\{\mu_i\}_{i=1}^{n-1} $ in the $e_n$-direction 
 such that 
  \begin{equation} \label{mu}
 \nabla_{e_i} e_n = -\mu_i e_i 
 \end{equation} 
 for all $i = 1, \ldots, n-1$.
\end{proof}

\subsection{Holonomic hypersurfaces of the Euclidean space}\label{holonomic}
For  more information on holonomic coordinates,  we refer to \cite{DT}, and for more details on twisted products we refer 
to \cite{Re} and  \cite{Rec}.
 At each point $p\in \Sigma^n$,  the tangent space $T_p\Sigma^n$ decomposes into  an orthogonal sum 
of the 
 eigenspaces $T_i$ of  the  symmetric shape operator $A(p)$ of respective   distinct eigenvalues   $\lambda_i, \ i =1,\ldots, s$, i.e.  $
 T_p\Sigma^n = \bigoplus_{i=1}^s T_i
 $.  \\
Let us consider a neighborhood  of a  point $p$ where  the planes $\{T_i\}_{i=1}^s$ define  a smooth distribution.
From  Codazzi  equations (cf. Appendix  \ref{cristoffel})  one shows that the distribution $T_i$ is integrable for each $i=1,\ldots, s$.\\
 Then, again, using   Codazzi equations, 
  the leaves of the distribution are totally umbilical. Recall that a submanifold is totally umbilical if 
  the second fundamental form equals $\langle X,Y\rangle H$, 
where $H$ is \textit{de facto} the mean curvature vector field of  the submanifold.
 When $rk(T_i) \geq 2$, the leaves of the distribution $T_i$ are  $rk(T_i)$-spheres  of $\mathbb{E}^{n+1}$ (see for example Appendix  \ref{spheres}).\\
The   set of the   $s$   orthogonal   foliations $\{\mathcal{T}_i\}_{i=1}^s$, obtained by the integration of the $T_i$-distributions, defines an orthogonal net on $\Sigma^n$ which 
is the principal curvature   net (cf.  \cite{Re} and  \cite{Rec}).\\
  $\Sigma^n$ is {\sl  holonomic}  
   if locally  $\Sigma^n$ is diffeomorphic to the product manifold $ P:= \prod_{i=1}^s \mathcal{T}_i$ and  there exists 
   an isometric parametrization of $X: \Omega \subset P  \longrightarrow \Sigma^n $ and the pullback  twisted metric $g$,   
which in terms of  elements  $(x_i)_{i=1,\ldots, n} \in P$, can be expressed by the following form   
\begin{equation}\label{metric}
 g = \sum_{i=1}^sv_i^2(x_1,\ldots, x_n) d\sigma_i^2 =  \sum_{i=1}^n \tilde v_i^2(x_1,\ldots, x_n) dx_i^2,
 \end{equation}
  where $d\sigma_i^2 $ are line metrics or  spherical metrics of each $\mathcal{T}_i, \ i=1, \ldots, s.$

\subsection{Integral curves of  $\nabla h$}
We now assume that the BCHs  are holonomic.
\begin{lem}\label{para3} 
Let $\Sigma^n$ be a  holonomic proper  BCH.  Let   $U_{t} $  be a level hypersurface on $\Sigma^n$ of mean curvature function  $h$.
Then $\Sigma^n$ is locally fibered over  $U_{t} $ by the integral curves of $\nabla h$. These fibers are planar lines in $NU_{t}$, i.e. for each $p\in U_{t}$,    $N_pU_{t}\cap \Sigma^n$ is a curvature line of principal curvature 
 $\lambda_n$ and    a  geodesic of $\Sigma^n$. 
\end{lem}
\begin{proof}\phantom{x}
Consider $\gamma$ a perpendicular curve to $U_0$, starting at a point $p\in U_0$ 
and denote the unit tangent vector field along the curve by  $T(s)$  which,  at $s=0$, is  $e_n:= \frac{ \nabla h}{|\nabla h|}$.\\
First, the    curves  $\gamma$ are   geodesics of $\Sigma^n$. Indeed,
 by the Codazzi equations and since $\lambda_n$ is constant on $U_0$, we have (cf. Appendix \ref{cristoffel}) : 
$$ ( \lambda_n-\lambda_i)\Gamma_{ni}^i =( \lambda_n-\lambda_i)(\log v_n)_{,i}=0, $$ 
 so, as $\lambda_n \neq \lambda_i$, for all $i=1,\ldots,n-1$
 $$v_{n,i} = 0, \ i = 1,\ldots, n-1.$$
Hence, the coefficient of $dx_n^2$ in the expression of the metric $g$, which is $v_n$, depends only on $x_n$. By a possible change of coordinates along the $x_n$-curvature line the 
metric $g = \sum_{i=1}^{n-1} v_i^2 d\sigma_i^2 + d\tilde x_n^2$. It is then clear that the $\tilde x_n$- curves are geodesics. \\
 Second,  geodesics that are curvature lines must be planar:
$$\frac{D}{ds}(T\wedge N) = \frac{D}{ds}T\wedge N + T\wedge\frac{D}{ds}N.$$
As $\gamma$ is a geodesic of $\Sigma^n$,   $\frac{D}{ds}T =\kappa N$,
and as $ \frac{D}{ds} N=-\lambda_n T,$ we obtain
 \begin{equation}\label{planConstant}
 \frac{D}{ds}(T\wedge N) = 0.
 \end{equation}
 Therefore, the integral curves of $\nabla h$  are planar and contained in the fixed  plane $T_0\wedge N_0$.\\
  \end{proof}
\subsection{Parametrization of a  holonomic BCH}

 \begin{lem}\label{para4} 
Let $\Sigma^n$ be a holonomic proper BCH.
Let $N_0$ be the unit normal vector field of $\Sigma^n$ restricted to 
 $U_{0}$  and  let  $e_{n} = \frac{\nabla h }{|\nabla h|}$. Then, locally, there exists in a neighborhood of $U_{0}$  a parametrization of  $\Sigma^n$,  in terms of the curvature line coordinates $x_1,\ldots, x_n$, given by 
\begin{equation}\label{surRev}
X(x_1,\ldots,x_n) = Y(x_1,\ldots, x_{n-1}) + \alpha(x_n)N_0(x_1,\ldots,x_{n-1}) + x_n e_{n}(x_1,\ldots, x_{n-1}),
\end{equation}
where $Y$ parametrizes $U_{0}$, and $\alpha$ is a smooth  function such that $(x_n,\alpha(x_n))$ parametrizes the planar $x_n$-curves.
\end{lem}
\begin{proof}
By Lemma \ref{para3}, there is a local parametrization of $\Sigma^n$  by curvature line coordinates 
\begin{equation}
\begin{array}{cc}
 X:&\Omega \subset \mathbb{R}^n\longrightarrow \Sigma^n \subset \mathbb{R}^{n+1} \\
&(x_1,\ldots,x_n) \mapsto X(x_1,\ldots,x_n)
\end{array}
.
\end{equation}
Since the principal curvature coordinates form an orthogonal net in Subsection \ref{holonomic}  and  from the orthogonal decomposition 
of its metric  \eqref{metric},  $U_{0}$ is parametrized by 
$$
Y(x_1,\ldots, x_{n-1}) := X(x_1,\ldots, x_{n-1},0).
$$
Again, using lemma \ref{para3}, 
for each point $p = Y(x_1,\ldots, x_{n-1})$, the integral curve of $\nabla h$ passing through $p$ which is a  $x_n$-curvature line,
is a planar curve  in the affine plane $P_p$ passing through $p$ and generated by the normal at $p$ to $\Sigma^n$, $N_0$  and 
$e_n =\frac{\nabla h }{|\nabla h|}$.
Locally near $U_{0}$ this curvature line is a graph in $P_p$ over the $\mathbb{R}_{e_n}$-axis. Thus,  there exist two functions $\alpha(x_1,\ldots, x_n) $ and  $\beta(x_n)$ such that :
\begin{equation}
X(x_1,\ldots,x_n) = Y(x_1,\ldots, x_{n-1}) + \alpha(x_1,x_2,\ldots, x_n)N_0(x_1,\ldots,x_{n-1}) +\beta( x_n ) e_{n}(x_1,\ldots, x_{n-1}).
\end{equation}
Note that, if  $(x_1,\ldots, x_n)$ is locally a curvature line coordinate system (or orthogonal coordinate system), so is
 the new coordinate system $\left(x_1,\ldots,x_{n-1}, \beta_n( x_n)\right)$ 
for a  smooth local one-to-one function $\beta_n$. So by abuse of notation, we  write the following 
\begin{equation}\label{surRev2}
X(x_1,\ldots,x_n) = Y(x_1,\ldots, x_{n-1}) + \alpha(x_1,x_2,\ldots,x_n)N_0(x_1,\ldots,x_{n-1}) +x_n  e_{n}(x_1,\ldots, x_{n-1}).
\end{equation}

Now, by orthogonality of the curvature coordinates : \begin{equation}   \label{product}
\langle X_{,i}, X_{,n}\rangle =0
\end{equation} 
for all $i \in \{1,\ldots, n-1\}$. Let $\lambda_{0i}$, $i \in \{1,\ldots, n-1\}$,  be the principal curvatures $\lambda_i$ on $U_0$, 
then differentiating  $X$ with respect to $x_i$ we obtain
\begin{equation}\label{X}
X_{,i}  =( 1 -  \alpha\lambda_{0i} )  Y_{,i} 
 +x_n e_{n,i}+ \alpha_{,i}N_0 =( 1 -  \alpha\lambda_{0i}-x_n\mu_{0i} )  Y_{,i} 
 + \alpha_{,i}N_0 , \quad i\in \{1,\ldots, n-1\}, 
\end{equation}

 and
 \begin{equation}\label{ALPHA}
  X_{,n}= \alpha_{,n} N_0 +  e_{n}.
\end{equation}
From equations \eqref{product}, \eqref{X} 
 and \eqref{ALPHA} one obtains $\alpha_{,n}\cdot \alpha_{,i} =0$. If $\alpha_{,n}$ is null, then the $x_n$-curves are straight, i.e. $\lambda_n =0$. 
This implies $\lambda_n=\frac{-nh}{2}=0$, i.e. $\Sigma^n$ is minimal,  which contradicts the properness of $\Sigma^n$.
Consequently $\alpha_{,i} = 0$.
Hence, $\alpha$ is a function of $x_n$ only.
\end{proof}

\subsection{ Second order ODE governing  the integral curves of $\nabla h$}\label{ODE0}
We saw in the previous  lemma that the function $\alpha$ depends only on the parameter $x_n$. 
We now show that $\alpha$ is the unique solution of an ODE with the  initial conditions $\alpha(0)=\alpha'(0)=0$,
 which  implies  that the level sets of $h$ are isoparametric (Lemma \ref{para5}).\\
Denote by $\alpha'$ the derivative of $\alpha(x_n)$ with respect to $x_n$. 

\begin{lem}\label{paraODE}
Using the notations of Lemma \ref{para4} the function $\alpha$ locally satisfies  a second order 
 ODE  of the form 
$ \alpha'' =  R(x_n,\alpha,\alpha')$ where $R(x,y,z)$ is a  rational function  
 which is smooth in a neighborhood of $0$. The solution  is  locally unique  with the 
initial condition $\alpha(0)=\alpha'(0)=0$.
\end{lem}
\begin{proof}
From Lemma  \ref{para4} the BCH $\Sigma^n$,  locally,  admits   a parametrization of the form :
\begin{equation}\label{lEquation}
X(x_1,\ldots,x_n) = Y(x_1,\ldots, x_{n-1}) + \alpha(x_n)N_0(x_1,\ldots,x_{n-1}) + x_n  e_{n}(x_1,\ldots, x_{n-1}).
\end{equation}
Differentiating  $X$, we obtain
$$X_{,i}  =( 1 -  \alpha\lambda_{0i} )  Y_{,i} 
 +x_n e_{n,i}, \quad i\in \{1,\ldots, n-1\},
 $$
 and
 $$
 X_{,n}= \alpha' N_0 + e_{n}. $$
 From the Codazzi  equation for $\Sigma^n$ \eqref{total}
there exists a  principal curvature  $\mu_{0i }$ on $U_0$ such that $$ e_{n,i}= -\mu_{0i} Y_{,i},$$
where $Y_{,i}\in TU_0$, for all $ i\in \{1,\ldots, n-1\}$.\\
  Equivalently, from  the Codazzi  equation for $\Sigma^n$ \eqref{coda},
 and   by holonomicity 
 $$[e_i,e_n] = \nabla_{e_i}e_n -  \nabla_{e_n}e_i = 0,$$
 then
 $$ \nabla_{e_i}(A)(e_n)- \nabla_{e_n}(A)(e_i) = A([e_i,e_n]) = 0.$$
 As $ \nabla_{e_i}(A)(e_n) = \nabla_{e_i}(Ae_n)- A( \nabla_{e_i} e_n)$ and
 $ \nabla_{e_n}(A)(e_i) = \nabla_{e_n}(Ae_i)- A( \nabla_{e_n} e_i)$,  we get
 $$\nabla_{e_i}(\lambda_n e_n) - \nabla_{e_n}(\lambda_i e_i)=0.$$
 And since $\nabla_{e_i}(\lambda_n)=0$, we obtain 
 $$\nabla_{e_i} e_n =- \frac{\nabla_{e_n} \lambda_i}{\lambda_i -\lambda_n}e_i$$
 which is equation \eqref{total}.

  Thus, we  deduce  that 
  \begin{equation}\label{metriqueEvolution}
  \begin{array}{ll}
        X_{,i}&   =( 1 -  \alpha\lambda_{0i} -x_n\mu_{0i})  Y_{,i}  := \beta_i Y_{,i} 
    \quad i\in \{1,\ldots, n-1\}, \\
      X_{,n}&= \alpha' N_0 + e_{n}.  
  \end{array}
  \end{equation}

The coefficients of the metric \eqref{metric} of $\Sigma^n$ are thus equal to 
   $$v_i ^2=\langle X_{,i},X_{,i}\rangle= v^2_{0i}( 1 -  \alpha(x_n)\lambda_{0i} -x_n\mu_{0i})^2\quad {\rm if}\  i\neq n$$
 and 
 $$v^2_n = 1+\alpha'^2$$
  which will be also denoted by  
 $g_{ii} = \beta_i^2 v^2_{oi}, \  i\in \{1,\ldots, n-1\}, \  g_{nn} = 1+\alpha'^2=\beta_n^2$.\\
 Let us investigate  the evolution of the principal curvature along an integral curve of $\nabla h$ defined by $\alpha$. 
  Since, from equation \eqref{planConstant}, the normal planes to $U_0$ are constant along integral curves, it is straightforward that 
 \begin{equation}\label{unitNormal}
N=(N_0 -\alpha'  e_{n})\frac{1}{\sqrt{1+\alpha'^2}}.
\end{equation}
 Along curvature lines of $U_0$  and by definition of the principal curvatures $\lambda_{0i}$ of $U_0$:
$$N_{0,i} =-\lambda_{0i}Y_{,i} \ \  i\in\{1,\ldots, n-1\}.$$
 Plug  these identities   into the derivatives of \eqref{unitNormal} 
 along curvature lines :
$$N_{,i} =N_{0,i} \frac{1}{\sqrt{1+\alpha'^2}} -\alpha'\mu_{0i} \frac{1}{\sqrt{1+\alpha'^2}}Y_{,i}=
-\left(  \lambda_{0i}  +   \alpha'\mu_{0i}    \right) Y_{,i}     \frac{1}{\sqrt{1+\alpha'^2}}.$$
Since the  variation of the normal of $\Sigma^n$ along  the $x_i$- curvature lines for $i\in\{1,\ldots, n-1\}$ is given by the identities 
$N_{,i} = -\lambda_i X_{,i} =-\lambda_i Y_{,i}\beta_i$, (see equations \eqref{metriqueEvolution})
we deduce that  the variation of the principal curvature along the $x_n$-curvature lines are :
\begin{equation}
\lambda_i =(\lambda_{0i}+\alpha'\mu_{0i})\frac{\gamma}{\beta_i} ,\quad 
 i\in \{1,\ldots , n-1\},{\rm where} \  \gamma = \frac{1}{\sqrt{1+\alpha'^2}}.
\end{equation}
First,  we obtain  
\begin{equation}\label{lambda1}
\sum_{i=1}^{n-1} \lambda_i  =  \frac{1}{\sqrt{1+\alpha'^2}}  \sum _{i=1}^{n-1} 
  \frac{\lambda_{0i}+\alpha'\mu_{0i}}{1 -\alpha\lambda_{0i}-x_n\mu_{0i} } . 
\end{equation}
Also differentiationg  along the  $x_n$-curvature line
$$N_{,n} =  (\frac{1}{\sqrt{1+\alpha'^2}})' N_0 -  e_{n}(\frac{\alpha'}{\sqrt{1+\alpha'^2}})' = -\lambda_n X_{,n} =-\lambda_n ( \alpha' N_0 + e_n)$$
 we simply obtain the standard equation:
\begin{equation}\label{lambda2}
 \lambda_n =  \frac{\alpha''}{(1+\alpha'^2)^{3/2}}.
\end{equation}
Since $nh+2\lambda_n=0,$ i.e.  
$\sum_{i\neq n} \lambda_i +3\lambda_n =0,$
and from equations \eqref{lambda1} and \eqref{lambda2}
$\alpha$ satisfies the following ODE of order 2
\begin{equation}\label{courbeIntegrale}
\frac{\alpha''}{1+\alpha'^2} =-\frac{1}{3}\sum_{i=1}^{n-1} \lambda_i  =  \sum _{i=1}^{n-1} 
  \frac{\lambda_{0i}+\alpha'\mu_{0i}}{1 -\alpha\lambda_{0i}-x_n\mu_{0i} } .
\end{equation}
The  2nd order ODE  is of the form 
$ \alpha'' =  R(x_n,\alpha,\alpha')$, where the rational function  $R$  equals 
$$R(x,y,z) = -\frac{(1+z^2)}{3}\sum_{i=1}^{n-1}\frac{\lambda_{0i}+z \mu_{0i}}{1-\lambda_{0i}y-x\mu_{0i}}.$$

$R$ is a smooth  function when  $(x,y) $ lies in a neighborhood of $(0,0)$. Hence, $\alpha$ is  locally the unique   solution with 
the initial condition $\alpha(0)=0$ by construction, and $\alpha'(0)=0$ since the $\alpha$-curve is tangent to $e_n$ at $0$. From equation \eqref{courbeIntegrale}  $\alpha$ is analytic.\\
\end{proof}

\subsection{ Level  sets of  $h$ }
\begin{prop}\label{para5}  The level sets of $h$ are isoparametric and the  integral curves of $\nabla h$ are congruent planar curves.
\end{prop}

\begin{proof} We are 
going to show that the principal curvatures of the $h$-level set $U_0$  are constant. Indeed 
From Lemma  \ref{paraODE}  the $\alpha$-curve is a function of  only  the $x_n$-coordinate. Hence all the $\alpha$-curves are congruent.  In particular,
from equation \eqref{courbeIntegrale}, from Lemma \ref{para1}, and as $h$ is constant on $U_0$,  for any two points $s,s'\in U_{0}$  and $x_n=t$ in a small neighborhood of $0$, we have
\begin{equation} \label{paraODEbis}
 \tilde R(s,t):= \sum_{i=1}^{n-1}\frac{\lambda_{0i}(s)+\alpha' (t)\mu_{0i}(s)}{1-\lambda_{0i}(s)\alpha(t)-t\mu_{0i}(s)}=
 \tilde R(s',t):=
  \sum_{i=1}^{n-1}\frac{\lambda_{0i}(s')+\alpha' (t)\mu_{0i}(s')}{1-\lambda_{0i}(s')\alpha(t)-t\mu_{0i}(s')}.
\end{equation}
Let $c_i(s):=\frac{1}{\lambda_{0i}}(s)$  and  $b_i(s) =\frac{\mu_{0i}}{\lambda_{0i}}(s)$. 
First, we show the uniqueness of the coefficients $(b_i,c_i)$
in the  case  where $\alpha,\alpha'$ and $t$ are algebraically independent.

Let us see 
$\alpha $ as a (convergent) power series of $\mathbb{C}[[t]]$. From its properties, 
$\alpha = h.t^2$ and  $\alpha' = k.t$, where $h$ and $k$ are invertible elements of  $\mathbb{C}[[t]]$ since their zero order coefficient 
is non-zero (cf. equation \eqref{lambda2}).  Then equation  \eqref{paraODEbis} becomes an equality in the
  \begin{equation} \label{ODE5}
\sum_{i=1}^{n-1}\frac{tkb_i(s)+1}{t^2h +tb_i(s) -c_i(s)} =
\sum_{i=1}^{n-1}\frac{tkb_i(s')+1}{t^2h +tb_i(s') -c_i(s')} = \frac{f}{g}.
 \end{equation}
$\frac{f}{g}$ is a rational function of $t$ in the extension field $ K=\mathbb{C}(h,k)$
where $f$ and $g$ are polynomials in $t$ 
 with coefficients in the field  $K$ and $deg(f)<deg(g)$. Equation \eqref{ODE5} then provides two decompositions of the same $\frac{f}{g}$.
\begin{lem} If there are  two  points $s_1,s_2\in U_0$ such that 
$\tilde R(s_1,t) = \tilde R(s_2,t)$ and there is no non-zero polynomial $P$ such that $P(t,k,h)=0$, then 
 $\left(b_i\left(s_1\right)),c_i\left(s_1\right)\right) = (b_{\sigma(i)}(s_2),c_{\sigma(i)}(s_2))$ for $i= 1,\ldots,n-1$ and some  permutation $\sigma \in \mathfrak{S}_{n-1}$.
\end{lem}
\begin{proof}

 Let us consider the first decomposition, for example, the LHS.
Consider any two polynomials of the form $ht^2+bt-c$ and $ht^2+b't-c'$ that are different, 
  i.e. such that   $(b,c) \neq (b',c')$: 
\begin{enumerate}
    \item 
    If $b=b'$, then 
  $c\neq c'$ and the polynomials are coprime. 
  \item If $ b\neq b'$, then the Euclidean algorithm  yields at the second step for remainder 
 $$r= -c'+\frac{c-c'}{(b-b')}\left(b' + \frac{(c-c')h}{(b-b')}  \right).$$
 \begin{enumerate}
 \item If $c=c'$,  then $r=c'$ which is non-zero.
 \item  If $c\neq c'$, then the $h$ coefficient is non-zero; hence $r\neq 0$. 

 \end{enumerate}
\end{enumerate}  
In all these cases, the remainder is a non-zero element of $K$; hence the denominators are coprime.
 In conclusion,   any  two polynomials of the form $ht^2+bt-c$ are  either equal or coprime.\\
 Second,  we consider a term of the LHS:
  \begin{equation} 
 \frac{tkb+1}{t^2h +tb -c}, 
 \end{equation}
 then,  either the fraction is irreducible, or
 $h= b^2k(1+kc)$ and the fraction reduces to the irreducible fraction
  \begin{equation} \label{ODE6}
 \frac{kb}{ht-kbc}. 
 \end{equation} 
Third, any considered  denominator is  either irreducible in $K$  or a product of 2 linear forms with 
non equal roots because  the discriminant $b^2+4ch\neq 0$ since $c\neq 0$.  The corresponding fraction can  then be explicitly and uniquely  split  as a sum of two 
fractions of the form:
 \begin{equation} \label{ODE7}
 \frac{tkb+1}{t^2h +tb -c}  = \frac{\delta}{t-\rho_1} +  \frac{\epsilon}{t-\rho_2}
\qquad  \delta,\epsilon,\rho_1,\rho_2 \in K. \end{equation}
 Grouping the fractions  having 
 identical denominators,  we obtain a partial fraction decomposition of $\frac{f}{g}$, i.e. a sum  of irreducible fractions
 with coprime denominators. By the uniqueness of the partial fraction
 decomposition ( cf. for example \cite{VDW}) the partial fraction
 decomposition of $\frac{f}{g}$ thus obtained  by the LHS is
  identical to  the partial fraction decomposition  given by the RHS.
\end{proof}
In general, suppose the principal curvatures of $U_0$ are 
non-constant.
Beforehand, one represents
the $\tilde R(s,t)$ as follows:
$$\tilde R(s,t) = 
\left(\sum_{i=1}^{n-1} \frac{1}{\alpha(t)-l_{i}(s)(t)}\right)+ 
\left(\sum_{i=1}^{n-1} \frac{b_i(s)}{\alpha(t)-l_i(s)(t)}\right)\alpha' = A_s(t,\alpha) + B_s(t,\alpha)\alpha'$$
where the linear forms $l_i(s)(t) =c_i(s)-b_i(s)t$, $i=1,\ldots,n-1$. Without reference to the parameter $s$, we define
$$A(t,\alpha)=\sum_{i=1}^{n-1} \frac{1}{\alpha(t)-l_i(t)}$$
and $$B(t,\alpha)=\sum_{i=1}^{n-1} \frac{b_i}{\alpha(t)-l_i(t)}.$$
Then, equations  \eqref{paraODEbis} or \eqref{ODE5} are true 
for a continuum of   $s$ (parametrizing for example a curve on 
$U_0$), with equal  $\tilde R(s,t)$, but
different $\{(b_i(s),c_i(s)) \}_{i=1,\ldots, n-1}$. Hence 
$$\alpha' = Q_s(t,\alpha)
$$ for distinct rational functions $Q_s(x,y)$ 
such that  $\frac{d }{ds}|_{s=s_1}Q_s(x,y)\neq 0$ in $\mathbb{C}(x,y)$.\\
From the preceding lemma, we can suppose that  there is a non-trivial polynomial $P(x,y,z)$ such that $P(t,\alpha,\alpha')=0$.
\begin{lem} Suppose that the principal curvatures of $U_0$ are non-constant and that there is a non-zero polynomial $P$ such that $P(t,\alpha,\alpha')=0$, then 
$\alpha$ extends to a non-constant  multi-valued holomorphic algebraic function on $\mathbb{C}.$ \end{lem}
\begin{proof}

Let us show that there is a non-trivial polynomial $\tilde P$ such that,
$\tilde P(t,\alpha)=0$. If $deg_zP(x,y,z)=0$, then $P(x,y,z)=P(x,y)\neq 0$, by hypothesis, so we can choose $\tilde P = P$. 
Suppose  now that $deg_zP(x,y,z)=1$, then $P(x,y,z) = P_0(x,y) + P_1(x,y)z$
and by substitution we get 
$$P(t,\alpha,\alpha') = P_0(t,\alpha) + P_1(t,\alpha)\alpha'=
P_0(t,\alpha) + P_1(t,\alpha)Q_s(t,\alpha)=0.$$
Differentiating   with respect to $s$, we see that $P_1(t,\alpha)=0$ with $P_1$ non-trivial
(by iteration, the result extends  for any degree of $P$  with respect to $z$).
Furthermore, the existence of a non-trivial $P$ such that $P(t,\alpha)=0$ implies that
$\alpha$ is an algebraic function which  extends to an $m$-valued holomorphic algebraic function
on $\mathbb{C}$. This, in turn, lifts to a holomorphic function  defined on a  compact  Riemann surface  $\mathfrak{R}$ and  with  values  in the Riemann sphere $\mathbb{S}^2=\mathbb{C}\cup \{\infty\}=\mathbb{CP}^1$. The analytic continuation of $\alpha$, also denoted by $\alpha$,  covers  $\mathbb{S}^2$ $d$-times,
where $d\geq 1$  ($\alpha$ is assumed not to be constant).
\end{proof}
Let us now derive a contradiction in case the hypotheses of  the  preceding lemma hold. Suppose 
\begin{equation}\label{equa10}
 A_{s_1}(t,\alpha) + B_{s_1}(t,\alpha)\alpha'=A_{s_2}(t,\alpha) + B_{s_2}(t,\alpha)\alpha'     
\end{equation}
  but with non identical polynomials
$A_{s_i}(x,y)$ or $B_{s_i}(x,y), i =1,2.$
Hence, there is  
a linear form  $l_i(s_1)$ different from the $l_j(s_2),\ j=1,\ldots, n-1$.
A  zero of the  function $\alpha-l_i$, for some $i\in \{1,\ldots,n-1\}$, is a pole of $A_{s_1}(t,\alpha)+ \alpha'B_{s_1}(t,\alpha)$. Since all values are taken by 
$\alpha$, there is at least a $t_0\in \mathbb{C}$ such that 
$\alpha(t_0)-l_i(t_0)=0$. By equation \eqref{equa10},  $t_0$ is also a pole of  the RHS, i.e. a  zero of $\alpha -l_i(s_2)$ (up to a possible reindexing of $i$), so 
\begin{equation}\label{egalite10}
 t_0 (b_i(s_2)-b_i(s_1)) = c_i(s_2)-c_i(s_1).
\end{equation}
A local Laurent series expansion  around $t_0$  shows that  the orders of the corresponding poles are equal, i.e. $n_i=n'_i$, and that  
$\alpha'(t_0) +b_i = \alpha'(t_0) +b'_i$. Hence, necessarily,  $b_i=b'_i$
and from \eqref{egalite10}, $c_i=c'_i$ ($t_0\neq 0$).
This  contradicts the hypothesis on $l_i(s_1)$ and $l_i(s_2)$. Consequently 
$\left(b_i\left(s_1\right)),c_i\left(s_1\right)\right) = (b_{\sigma(i)}(s_2),c_{\sigma(i)}(s_2))$ for $i= 1,\ldots,n-1$ and  for some  permutation $\sigma \in \mathfrak{S}_{n-1}$.

We conclude that   
the principal curvatures of the $h$-level set $U_0$  must be  constant, and since its normal bundle is flat from Lemma \ref{para1}, 
$U_{0}$  is isoparametric (cf. introduction).
\end{proof}
\subsection{Proof of  the first part of Theorem \ref{letheoreme1} }
 Lemmas  \ref{para1}  till  \ref{paraODE} prove  Theorem \ref{letheoreme1}  except for the symmetries of $\Sigma^n$.
 Let $G$ be the subgroup  (possibly trivial) of the ambient isometries of $\mathbb{E}^{n+1} $ that preserve $U_0$.
Let $g\in G$  and $p\in \Sigma^n$. Then $p $ belongs to an integral curve of $\nabla h$
that cuts $U_0$ at point $p_0 \in U_0$. Let $q_0 :=g(p_0)\in U_0$. The normal plane  $N_{p_0}U_0$ is sent by the isometry $g$ 
to $g(N_{p_0}U_0)= N_{q_0}U_0$, and  the normalized mean curvature at $p_0$, $N(p_0)$, 
is sent to   $ N(q_0)$. Hence  we have also $g(e_n(p_0)) = e_n( q_0)$. By uniqueness of the ODE \eqref{courbeIntegrale} with given
initial conditions and by Lemma \ref{para5},
 the  integral curve $N_{p_0}U_0 \cap \Sigma^n$ is sent to
$g(N_{p_0}U_0 \cap \Sigma^n) = N_{q_0}U_0 \cap \Sigma^n.$ 
Consequently, $g(\Sigma^n)=\Sigma^n$.
  \section{Proof of  the last  part of Theorem \ref{letheoreme1} and Theorem \ref{theoreme2} }\label{constructionBCH}
  Conversely, Subsection \ref{ODE0} gives a method to generate proper 
   $BCH$ by evolution of a given  isoparametric codimension $2$ submanifold 
  $U_0 \subset \mathbb{E}^{n+1}$ in the direction of a unit normal  vector field to $U_0$.
     First let us illustrate the method in case the isoparametric  submanifold $U_0\subset \mathbb{E}^n\subset 
     \mathbb{E}^{n+1}$.
       \begin{prop} \label{SigmaU} Let  $U_0$ be  an  isoparametric and  holonomic hypersurface of
        $\mathbb{E}^n \subset \mathbb{E}^{n+1}$. There exists a proper  
 holonomic BCH   
containing $U_0$   and whose  $n$-plane containing $U_0$ is a plane of  symmetry.
\end{prop}

 \begin{proof}
We refer to notations of Subsection \ref{ODE0}.
We consider the parametrization \eqref{surRev2} where $ e_n$ is  a unit normal vector field to the n-plane containing $U_0$. 
The principal curvatures $\mu_i$, i.e. the eigenvalues of the shape operator $A_{ e_n} $, are zero. Hence, 
the second order ODE  \eqref{courbeIntegrale} which characterizes a BCH  becomes:
\begin{equation}\label{equaDiff}
\frac{\alpha''}{1+\alpha'^2}=-\frac{1}{3}\sum_{i=1}^{n-1}\frac{\lambda_{0i}}{1-\lambda_{0i}\alpha}= R(\alpha).
\end{equation}
The solution can be explicitly given by double integration.\\
 Let $\alpha' := u $ then 
 $$\alpha'' =\frac{du}{dx}= \frac{du}{d\alpha}\frac{d\alpha}{dx} = (1+u^2)R(\alpha).$$ 
 Hence,
 $$\frac{udu}{(1+u^2)}  = R(\alpha) d\alpha
 \ {\rm  and}\   \log\sqrt{1+u^2}=\int_.^\alpha R(t) dt +C_1$$
 Replacing $R$ by its expression, we obtain
 $$1+\alpha'^2 = C_2\prod_{i=1}^{n-1} (1-\lambda_{0i}\alpha)^{2/3}.$$
 The initial conditions
$\alpha(0) = \alpha'(0)=0$ implies $C_2=1$. Choose the orientation of the normal $N$ such that $h<0$. Then $ \prod_{i=1}^{n-1}  (1 -\lambda_{0i}\alpha)^{2/3} -1$
is nonnegative for $\alpha$ positive and
\begin{equation}\label{alphaPrime}
 \alpha' = \pm\sqrt{ \prod_{i=1}^{n-1}  (1 -\lambda_{0i}\alpha)^{2/3} -1},
 \end{equation}
  so 
   $$ \frac{d\alpha}{  \sqrt{ \prod_{i=1}^{n-1}  (1 -\lambda_{0i}\alpha)^{2/3} -1}} = \pm dx.$$
 A second integration yields an expression of the inverse  of the function $x = f(\alpha)$ and the value of $\lambda_n$
 from equation \eqref{lambda2}:
 \begin{equation}\label{meridienne2}
x = C_3 \pm \int_0^\alpha  \frac{dt}{  \sqrt{ \prod_{i=1}^{n-1}  (1 -\lambda_{0i}t)^{2/3} -1}} \ {\rm and} \
\lambda_n = 
\frac{1}{3}\left(\sum_{i=1}^{n-1} \frac{1}{1-\lambda_{0i}\alpha}\right) \prod_{i=1}^{n-1}(1-\alpha\lambda_{0i})^{-1/3}.
 \end{equation}
 
The initial condition
$\alpha(0)  =0$ implies $C_3=0.$
  Notice that the solution $x$ of equation \eqref{meridienne2}  is an   integral which is equivalent to $\int_0^\cdot \frac{dt}{  \sqrt{ t}}$, hence, integrable if 
$\sum_1^{n-1} \lambda_i \neq 0$ (
$\sum_1^{n-1} \lambda_i= 0$  then  $\lambda_n=0$ and $\Sigma^n$ would be  minimal). Notice also 
that the function $\alpha$ is similar to  an elliptic function in the sense that its inverse is an integral of an algebraic function.
 \end{proof} 
In the case where $U_0$ is a $(n-1)$-sphere invariant by $O(n)$,  the principal   curvatures $\lambda_{0i}$'s are all equal and we obtain
 BCHs that are   stable by  $O(n)$,        
 that are  of catenoidal type and that are described in details in \cite{MOR} (see also \cite{N}).
\begin{figure}[h!]\label{fig}
\begin{center}
\includegraphics[scale=0.8]{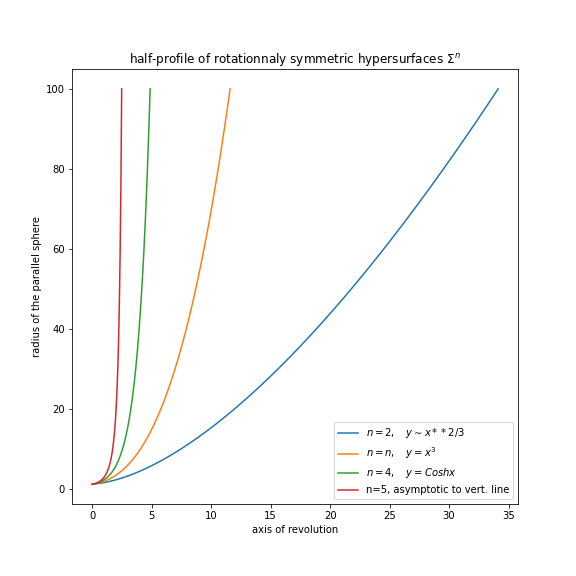} 
\end{center}
\end{figure} 
 \newpage
Similarly, if $U_0 $ is a product of spheres $\mathbb{S}^p(r_1)\times \mathbb{S}^q(r_2)$,  $p+q=n-1$ 
then $U_0 \subset \mathbb{S}^n\left(\sqrt{r_1^2 + r_2^2}   \right) \subset \mathbb{E}^{n+1}$ is invariant by 
$O(p+1)\times O(q+1)$, hence $\Sigma^n$ is also invariant by $O(p+1)\times O(q+1)$. The parametrization in \eqref{surRev2}, where
$Y$ parametrizes $U_0$,  $N_0$ is the normalized mean curvature vector field 
of $U_0$-which  is the unit normal  vector field to $ \mathbb{S}^n\left(\sqrt{r_1^2 + r_2^2}\right)$, $e_n$ is orthonormal  to $N_0$ in $NU_0$ 
and  $\alpha$ satisfies  the solvable  ODE \eqref{courbeIntegrale},    determines the proper BCHs  that
 were already 
described in  \cite{MOR}.\\ 
\\
These are the only cases   to be considered as we shall see now.
The fact that the level sets of the mean curvature function $h$ of a holonomic BCH are isoparametric,  is very restrictive. We refer to  \cite{CR} and \cite{T} for the following assertions  on isoparametric
submanifolds of the Euclidean space.
 Let $l$ be the number of distinct principal curvatures
of a codimension 2 isoparametric submanifold  $U_0\subset \mathbb{E}^{n+1}$. If $l=1$, then $U_0$ is a piece of a sphere or  a plane.
If $l=2$, then $U_0$ is a product of spheres or spheres and planes.  In both  cases, we recover the ones of Section \ref{constructionBCH} 
already described in \cite{MOR}.
 If $l\geq 3$, then $U_0$ is not holonomic (cf. for example \cite{DT}). Hence, from Theorem \ref{letheoreme1} 
all the proper holonomic BCHs have been described, but none of them is a proper BBH 
because the ODE  \eqref{courbeIntegrale} and  the  normal biharmonic equation \eqref{System3}  are not compatible. 
This was already proved in \cite{MOR} thus Theorem \ref{theoreme2} follows.
 \section{A question}
 Although all known examples of proper BCH are holonomic, the holonomic hypothesis is strong. 
 So the question is: can we extend the method of construction of proper BCH  described above  to non-holonomic isoparametric codimension 2 
 submanifolds?  The simplest  candidate  is  a projective tube in $ \squatre$
  i.e. a tubular neigbourhood of an embedding of $\mathbb{RP}^2\subset \mathbb{S}^4,$ 
that is  a non-holonomic  isoparametric codimension 2 submanifold of $\mathbb{E}^5$ that we describe now.
\begin{rem}
The embedding of $\mathbb{RP}^2$  is given in \cite{CA}  without proof so let us give an explicit  description.
  $\mathbb{RP}^2$ is embedded   into a round sphere of radius $A$, by  $X:\rpdeux \longrightarrow \mathbb{S}^4(A)\subset \mathbb{E}^5$, a Veronese-type embedding with 
   unknown coefficients $a,b,c,d,e,f$ and $A$ where 
  $$ X: (x,y,z) \mapsto (yz,xz,xy, ax^2+by^2+cz^2,dx^2+ey^2+fz^2).$$
 Since the codomain of $X$ is $ \squatre(A) $ then  necessarily $a=A\cos\alpha,d=A\sin\alpha,b=A\cos\beta,e=A\sin\beta,c=A\cos\gamma,f=A\sin\gamma$
  with $\alpha =0,\beta = -2±\pi/3,\gamma =2\pi/3$. This  forces $A= 1/\sqrt{3}$ and 
  
 yields
    $$X(x,y,z) = (yz,xz,xy, \frac{1}{\sqrt{3}}(x^2-(y^2+z^2)/2),  \frac{1}{2}(-y^2+z^2) ).$$

  Parametrize with $x=\cos\theta,y=\sin\theta\cos\phi$ and $z=\sin\theta\sin\phi$, we obtain:

\begin{equation}
X (\theta,\phi)=
\frac{\sqrt{3}\sin^2\theta }{2}
\left(
\begin{array}{c}
e^{-2i\phi}\\
0\\
0
\end{array}
\right)+
\frac{\sqrt{3}\sin 2\theta }{2}
\left(
\begin{array}{c}
0\\
e^{i\phi}\\
0
\end{array}
\right)+
\frac{(3\cos^2\theta -1) }{2}
\left(
\begin{array}{c}
0\\
0\\
1
\end{array}
\right).
\end{equation}
\end{rem}



\section{Appendix:  Codazzi's equations   and  curvature line coordinates  }\label{cristoffel}

One deduces directly from Codazzi's equations two key properties  related to the 
principal curvatures of a hypersurface $\Sigma^n\subset \mathbb{E}^{n+1}$. 
  Recall that  the shape operator in the normal direction $N$, $A_N$,  decomposes $T\Sigma^n$ into distributions 
  $T_i,$ where   $T_i = ker(A_N(p)-\lambda_i Id)$ and  $\lambda_i $ 
   is the corresponding eigenvalue -or principal curvature-   of multiplicity $rk(T_i)=m_i.$   We will select  an open
    subdomain  where  the multiplicities $m_i,\   i=1,\ldots, s,$ of the corresponding 
   eigenvalues  are constant.  Then we can split   the tangent bundle $T\Sigma^n$   into  a family of integrable 
   distributions  $T_i\subset T\Sigma^n $. 
 Let 
 $X,Y \in T\Sigma^n $ two tangent vector fields and let $A$ be the shape operator of $\Sigma^n$, then Codazzi's equation is 
  $$\nabla_X(A)Y  =   \nabla_Y(A)X,$$ where 
  $\nabla_X(A)Y  = \nabla_X(AY) - A(\nabla_X Y).$
 \begin{enumerate}
 \item   Let  $T_\lambda$ be  the  distribution  which corresponds to the  eigenvalue $\lambda$ of the shape operator $A$  of $\Sigma^n\subset \mathbb{E}^{n+1}$
   of  locally constant rank larger than 2, then it is integrable 
   and defines a curvature surface $S(\lambda)$ on which $\lambda$ is constant .\\
   Indeed, let
$X,Y \in T_\lambda$ be linearly independent  vector fields in the given distribution $T_\lambda$. Then from  the  Codazzi equation:
 
\begin{equation}\label{coda} 
 \nabla_X(A)Y = \nabla_Y(A)X,
 \end{equation}
so
 $$(X\lambda)Y -(Y\lambda)X = (A-\lambda Id)[X,Y].$$
 As the LHS belongs to $ker(A-\lambda Id) $ and RHS belongs to 
 $Im(A-\lambda Id)$,  $X(\lambda)=Y(\lambda) = 0$ and $[X,Y] \in T_\lambda$.
  Thus, from Frobenius, the distribution associated to the principal eigenvalue $\lambda$ is integrable into a curvature surface    $S(\lambda)$ onto which   $\lambda$ is constant.
 \item   $S(\lambda)$ is totally umbilical, i.e. the second fundamental form of $S(\lambda)$ is equal to 
 $$II(X,Y) = \langle X, Y\rangle H, \forall X,Y \in TS(\lambda), $$
 where $ H$ is the mean curvature vector field of $S(\lambda)$.\\
 Indeed, 
 let $X$ be a  vector field in $TS(\lambda)$ and let $\eta$ be  a vector field normal to $S(\lambda)$ 
 such that $A\eta =\mu\eta$ and $\mu \neq \lambda.$ Then
  $$\nabla_X(A)\eta = \nabla_\eta(A)X.$$ 
  Expanding 
  $$(\mu Id -A)(\nabla_X\eta) =
  (\eta\cdot\lambda)X- (A-\lambda Id) (\nabla_\eta X) - (X\cdot\mu)\eta.$$
  The projection  operator onto $TS(\lambda)$ commutes with $A$ hence 
  $$(\mu Id -A)(\nabla_X\eta)^{TS(\lambda)} =
  (\eta\cdot\lambda)X- (A-\lambda Id) (\nabla_\eta X)^{TS(\lambda)} = (\eta\cdot\lambda)X.$$
 Hence, 
\begin{equation}\label{total}
 (\nabla_X\eta)^{TS(\lambda)} =\frac{(\eta\cdot \lambda) X}{\mu - \lambda}.
 \end{equation}
Let   $\{e_k\}_{k=1,\ldots ,l}$ be  an orthonormal  basis of $TS(\lambda)^\perp\cap T\Sigma^n$ in the principal curvature directions,
so that $(\nabla_{e_k} N )^{TS(\lambda)}= -\lambda_k e_k,$
where $N$ is a unit normal  vector  field to  $T\Sigma^n$, then the  second fundamental form of $S(\lambda) \subset \mathbb{E}^{n+1},$ 
according to equation \eqref{total}, is  equal to
$$II(X,Y) = \sum_{k=1}^n - \langle \nabla_{X} e_k,Y\rangle e_k -  \langle \nabla_{X}N,Y\rangle N
=     \sum_{k=1}^n \frac{e_k\cdot\lambda}{\lambda-\lambda_k }\langle X,Y\rangle e_k  - \lambda  \langle X,Y\rangle  N.      $$
Consequently, 
$$ II(X,Y) = \langle X,Y\rangle H, \quad  H= \sum_{i=1}^l  \frac{(e_i \cdot \lambda)}{\lambda_i - \lambda} e_i -
 \lambda N.  $$
where $H$ is the mean vector field of $S(\lambda)\subset \mathbb{E}^{n+1}$. $S(\lambda)$ is thus totally umbilical (see next Appendix \ref{spheres}).
\end{enumerate}
 
  Let 
 $\Sigma^n$ be parametrized by  an  immersion 
\begin{equation}
\begin{array}{cc}
 X:&\Omega \subset \mathbb{R}^n\longrightarrow \Sigma^n \subset \mathbb{R}^{n+1} \\
&(x_1,\ldots,x_n) \mapsto X(x_1,\ldots,x_n),
\end{array}
\end{equation}
where $\Omega$ is an open domain, and
$(x_1,\ldots, x_n)$ form a system of orthogonal coordinates. Then the metric $g$ in these coordinates is diagonalized as follows:
\begin{equation}
g_{ij} =\delta_{ij}E_i= \delta_{ij}v_i^2 .
\end{equation}
 The Christoffel Symbols of the metric $g$ are:
 $$\Gamma^i_{jk} = \frac{g^{ii}}{2}\left( g_{ki,j} +  g_{ij,k} -g_{jk,i}  \right) = \frac{1}{2E_i}\left( \delta_{ki}E_{k,j}  + \delta_{ij}E_{i,k} -\delta_{jk}E_{j,i}  \right). $$
If $i\neq j\neq k$  then $\Gamma^i_{jk} = 0$.\\
Hence, the only possible  nonzero Christoffel symbols up to permutation on lower indices are:\\
\begin{itemize}
    \item If $i = j \neq k$,  then
    $
    \Gamma^i_{ik} =  \frac{E_{i,k}}{2E_i} =(\log v_i)_{,k}.
    $
    \item If $j= k \neq i$, then $\Gamma^i_{jj} =  -\frac{E_{j,i}}{2E_i}=-(\log v_j)_{,i}\frac{v_j^2}{v_i^2}. $
\item If $j= k = i$, then $\Gamma^i_{ii} =  \frac{E_{i,i}}{2E_i}=(\log v_i)_{,i}$.
\end{itemize}
In the frame of this coordinate system where 
$e_i = \partial_i, \ i=1,\ldots, n$ the connection on $T\Sigma^n$ is given by the Cristoffel symbols 
\begin{equation}\label{nabla}
\nabla_{e_i}e_k = \Gamma_{ik}^\alpha e_\alpha =\Gamma_{ik}^i e_i +\Gamma_{ik}^k e_k.
\end{equation}

In the holonomic case, when the curvature lines form a coordinate system,
 the shape operator $A$
 is diagonalized and related to the second fundamental form $II$  as follows
$$ A_i^j = \delta_i^j \lambda_i\ {\rm and}\ II_{ij} = A^\alpha_i g_{\alpha j}=  \delta_{ij}\mu_i,$$
where $\{\lambda_i\}_{i=1}^n$ are  the  principal curvatures
and  $\mu_i = \lambda_iv_i^2$.
The Codazzi equations   are equivalent to the following  system of
$n(n-1)$ equations:
\begin{equation}\label{codazzi}
\lambda_{i,j}-(\lambda_j -\lambda_i) \Gamma^i_{ij}=0,\quad \forall i\neq j\quad  i,j \in \{1,\ldots, n\}.
\end{equation}

 \section{Appendix  : totally umbilical submanifolds of $\mathbb{E}^{n+1}$  }\label{spheres}
Let us recall that 
a submanifold $\Sigma^n$  is umbilical at $p\in \Sigma^n$ if there is a unit normal  field $N(p)\in N_p\Sigma^n$ and curvature 
$\lambda\in\mathbb{R}$  such that:
$$II(p)(X,Y) = \langle X,Y\rangle \lambda(p) N(p)\qquad  \forall X,Y\in T_p\Sigma^n. $$
By definition, $\Sigma^n$ is   totally umbilical if   it is  umbilical at each point of $\Sigma^n$.
 \begin{lem}\label{lemma6}
 A totally umbilical hypersurface  $\Sigma^n \subset \mathbb{E}^{n+1}$ is a piece of a hypersphere or a hyperplane.
 \end{lem}
 \begin{proof}
 Let $N$ be a unit normal field to $\Sigma^n$.
 Let $\gamma(s)$ be  a curve in  $\Sigma^n$ and $T$ be the unit  vector field tangent  to $\gamma(s)$ then as $\frac{DN}{ds} $  and  
  $ \nabla_T N$ are both   orthogonal to $N$, we obtain
 
 \begin{equation}\label{normal1}
D_T N=  \frac{DN}{ds} = \nabla_T N \end{equation}

and since  any direction in $T\Sigma^n$  is a principal curvature direction,
 \begin{equation}\label{normal2}
 \nabla_T N= -\lambda T.
\end{equation}
From Appendix \ref{cristoffel} 1,   $\lambda$ is constant and integrating  Equation \eqref{normal2}, 
 $$N(s)-N(0) =-\lambda \int  \frac{DX}{ds}  ds = -\lambda\left( X(s)-X(0)\right).$$
 Hence, if $\lambda \neq 0$ we obtain
 $$X(s) = X(0) - \frac{1}{\lambda}(N(s)-N(0) := C(0) -\frac{1}{\lambda}N(s). $$
 Thus  $\Sigma^n \subset \mathbb{S}^{n}(C(0),\frac{1}{\lambda}).$
 If $\lambda =0$, $\Sigma^n $ is a piece of a  hyperplane.
 \end{proof}
 More generally
 \begin{lem}\label{7}
 A totally umbilical submanifold   $\Sigma^k \subset \mathbb{E}^{n+1}$, $k>1$  with the extra condition  that the 
 mean curvature is parallel  is a piece of a round sphere.
 \end{lem}
 \begin{proof}
 Let $(T_1,\ldots, T_k) $ be an orthonormal basis at some point extended to orthonormal frames  around some point.
 The unit normal vector field  $N =\frac{H}{h}$  seen as a section of the normal bundle of $ \Sigma^k $ is  parallel, hence
 for any $i = 1,\ldots, k$,  $(D_{T_i}N)^\perp=0$ 
 from which we deduce equation \eqref{normal2}  and, as in Lemma \ref{lemma6},  that  
  $\Sigma^k\subset \mathbb{S}^{n}(C(0),\frac{1}{\lambda})$.\\
  On the other hand, consider the $(k+1)$-vector field $ W:=T_1\wedge\cdots\wedge T_k\wedge N$;
 then    
 $$D_{T_i} W = \sum_{\alpha =1}^k T_1\wedge\cdots\wedge  D_{T_i}T_\alpha \wedge\cdots \wedge T_k\wedge N+
  T_1\wedge\cdots\wedge T_k\wedge D_{T_i}N,\quad i\in\{1,\ldots,k\}.$$
  The last term of the RHS is zero because $D_{T_i}N = -\lambda T_i$,  and 
  the first term is also zero since $D_{T_i}T_j \perp T_j$.
 Hence $D_{T_i} W = 0$ and  $W$ is constant on $\Sigma^k$ which means $\Sigma^k$ is in a $(k+1)$-plane $P$.
Finally  $\Sigma^k$ is in $S^n\cap P$ which is a round $k$-sphere. \end{proof}
Furthermore,  it is shown (cf. for example Prop. 1.19 in \cite{DT}) 
 that the mean curvature vector field of any umbilical submanifold of dimension  at least two is  parallel. \\
 We then deduce from Lemma \ref{7}:
 \begin{cor}
 A totally umbilical submanifold   $\Sigma^k \subset \mathbb{R}^{n+1}$, $k>1$  is a piece of  a sphere or a  plane.
 \end{cor}
 Note that    a helix with constant curvature and nonzero torsion  is umbilical,
  but the mean curvature vector is not parallel and the curve is not a circle unless the torsion is  null, i.e. unless the normal bundle is  flat.

\textbf{Acknowledgements:} The authors would like to thank Cezar Oniciuc for carefully reading our paper and for his fruitful comments and suggestions.

\noindent{\fontsize{8}{8} \selectfont INSTITUT DENIS POISSON, CNRS UMR 7013, UNIVERSITE DE TOURS, UNIVERSITE D'ORLEANS, PARC DE GRANDMONT 37200 TOURS, FRANCE}\\
{\it Email address: } {\tt  hiba\_bibi95@hotmail.com}\\
\noindent{\fontsize{8}{8} \selectfont  INSTITUT DENIS POISSON, CNRS UMR 7013, UNIVERSITE DE TOURS, UNIVERSITE D'ORLEANS, PARC DE GRANDMONT 37200 TOURS, FRANCE}\\
{\it Email address: } {\tt  marc.soret@idpoisson.fr}\\
\noindent{\fontsize{8}{8} \selectfont  UNIVERSITE PARIS EST CRETEIL, CNRS LAMA, F-94010 CRETEIL, FRANCE}\\
{\it Email address: } {\tt  villemarina@yahoo.fr}

 \end{document}